\documentclass{article}
\usepackage[utf8]{inputenc}
\usepackage{amsmath,amssymb,amsfonts,amsthm}
\usepackage[english]{babel}
\usepackage{parskip}
\usepackage{hyperref}
\usepackage{tikz}
\usepackage{graphicx}
\usepackage{float}
\usepackage{subcaption}
\usepackage{fancyhdr} 
\usepackage{geometry}
\usepackage{setspace}
\usepackage{csquotes}
\usepackage{comment}
\usepackage{authblk}

\newtheorem{lema}{Lemma}[section]
\newtheorem{prop}{Proposition}[section]
\newtheorem{teo}{Theorem}[section]
\newtheorem{cor}{Corollary}[section]
\newtheorem{claim}{Claim}[prop]
\newtheorem{claimsec}{Claim}[section]
\newtheorem{conj}{Conjecture}[section]

\newcommand{\li}[1]{\overline{#1}}
\newcommand{\Li}{\mathcal{L}} 

\title{Chen--Chvátal Conjecture for Graphs of Diameter 3}
\author{Martín Matamala\thanks{DIM-CMM, CNRS-IRL 2807, Universidad de Chile, Chile. E-mail: mar.mat.vas@dim.uchile.cl}
~and 
Luciano Villarroel-Sepúlveda\thanks{DIM, Universidad de Chile, Chile. E-mail: luciano.gabriel.vs@gmail.com}}
\date{December 2025}

\begin{document}

\maketitle

\begin{abstract}
In 2008, Chen and Chvátal conjectured that in every finite metric space of $n$ points, there are at least $n$ distinct lines, or the whole set of points is a line. This is a generalization of a classical result in the Euclidean plane. The Chen--Chvátal conjecture is open even in metric spaces induced by connected graphs. In 2018, it was asked by Chvátal whether graphs of diameter three satisfy the conjecture. In this work, we find all graphs of diameter three having fewer lines than vertices. As a direct consequence, we prove that graphs of diameter three satisfy the Chen--Chvátal conjecture.
\end{abstract}

\section{Introduction}
It is well known that a set of $n$ non--collinear points in the Euclidean plane determines at least $n$ distinct lines. This fact is a corollary of a combinatorial theorem of de Bruijn and Erd\H{o}s \cite{DBE1948}, and it is often referred to as the de Bruijn--Erd\H{o}s theorem, despite having been proved earlier in \cite{erdos4065Sol}. In 2008, Chen and Chvátal conjectured that the de Bruijn--Erd\H{o}s theorem can be extended to arbitrary finite metric spaces with an appropriate definition of line \cite{Chen-Chvatal2008}. Given three distinct points $x,y,z$ in a metric space $(V,d)$, we say that $z$ lies between $x$ and $y$ whenever $d(x,y)=d(x,z)+d(z,y)$, and we denote this by $[xzy]$. In this setting, the line generated by distinct points $x,y\in V$ is defined as:
$$
\overline{xy}=\{x,y\}\cup\{z\in V: [zxy]\, \lor \,   [xzy]\, \lor \,[xyz]\}
$$    
Due to the symmetry of the distance function $d(\cdot,\cdot)$, the line $\li{yx}$ is equal to the line $\li{xy}$. If the whole set of points is a line, we say that the metric space has a universal line. Therefore, the Chen--Chvátal conjecture can be stated as follows.
\begin{conj}[Chen--Chvátal]
Every finite metric space of $n$ points without a universal line determines at least $n$ distinct lines.  
\label{CCconjecture}
\end{conj}
It was recently shown that in every metric space with $n$ points and no universal line there are at least $\Omega(n^{2/3})$ lines \cite{Huang2025}, improving the previously best lower bound of $\Omega(\sqrt{n})$ \cite{Aboulker2016Lines}. 

Every connected graph induces a metric space in which the distance between two vertices is the minimum number of edges in a path joining them. We shall refer to the lines of a metric space induced by a connected graph $G=(V,E)$ simply as lines of $G$. The set of lines will be denoted by $\mathcal{L}(G)=\{\overline{xy}:x,y\in V, x\neq y\}$ and the number of lines by $\ell(G)=|\mathcal{L}(G)|$. The Chen--Chvátal conjecture has been widely studied in graph metrics, and several classes of graphs have been shown to satisfy the conjecture \cite{Aboulker2015DistanceHereditary, Beudou2015Chordal, Aboulker2018, Beudou2019Bisplit, Schrader2020q-q4,  Aboulker2022HHFree, Montejano2024, Matamala2025}. A class of connected graphs $\mathcal{C}$ is said to satisfy the Chen--Chvátal conjecture if every graph in $\mathcal{C}$ either has a universal line, or has at least as many lines as vertices.

In 2018, Aboulker, Matamala, Rochet and Zamora found all graphs $G=(V,E)$ in a super class of chordal and distance-hereditary graphs that do not satisfy the following inequality
\begin{equation}
\ell(G)+br(G)\geq |V|  
\label{Eq:CCbridges}
\end{equation}
where $br(G)$ is the number of bridges of $G$ \cite{Aboulker2018}. Since any bridge generates a universal line, connected graphs that satisfy inequality (\ref{Eq:CCbridges}) satisfy the Chen-Chvátal conjecture. In 2020, Matamala and Zamora found all connected bipartite graphs that do not satisfy (\ref{Eq:CCbridges}) \cite{Matamala2020}. In recent works, a similar question has been studied. Given a class of connected graphs $\mathcal{C}$, which graphs $G=(V,E)$ in $\mathcal{C}$ do not satisfy that $\ell(G)\geq |V|$ ? Answering this question will shed light on whether the class $\mathcal{C}$ satisfies the Chen--Chvátal conjecture or not, however, this problem is also of intrinsic interest. The answer may help deepen the understanding of configurations that determine a low number of lines in graphs, regardless of whether the graph has a universal line or bridges. In 2025, Matamala, Peña and Zamora found all locally connected graphs that have fewer lines than vertices \cite{Matamala2025} and most recently Matamala found all graphs of diameter two having fewer lines than vertices \cite{Matamala2025Diam2}. In this work, we find all graphs of diameter three with fewer lines than vertices. As a direct consequence, we obtain that graphs of diameter three satisfy the Chen--Chvátal conjecture. This, in fact, corresponds to Problem 2 in Chvátal's 2018 survey \cite{Chvatal2018Sur}, which, in the best of our knowledge, remained open until now. In 2024, it was shown by Montejano that graphs of diameter three with girth and minimum degree at least four have at least as many lines as vertices \cite{Montejano2024}.  It is also known that graphs of diameter three with $n$ vertices have $\Omega(n^{4/3})$ lines \cite{Aboulker2016Lines} and therefore that the number of graphs of diameter three with fewer lines than vertices is finite up to isomorphism. It turns out that there are only 14 such graphs, up to isomorphism.

\section{Preliminaries}\label{sec:Prelim}
In this section we present some useful lemmas that will help the reader become familiar with lines in graphs of diameter three. We also present 14 graphs of diameter three having less lines than vertices and offer some intuition about the main result.
Given a connected graph $G=(V,E)$ and three distinct vertices $x,y,z$, the relation of lying in between has the following characterization. $[xzy]$ if and only if there exists an $x$-$y$ path of length $d(x,y)$ containing $z$. From the definition of line, we also have that $[xyz] \Longrightarrow x\in  \overline{yz}\,\wedge\, y\in \overline{xz} \,\wedge\, z\in \overline{xy}$. These facts will be useful in determining which vertices belong to a given line. Some lines in graphs are particularly easy to determine. For example, the line generated by a bridge consists of the whole set of vertices. The line generated by vertices whose distance equals the diameter of the graph also admits a simple description. In the following lemma, we state this explicitly for graphs of diameter three.

\begin{lema}
Let $G=(V,E)$ be a graph of diameter 3 and $x,y\in V$
such that $d(x,y)=3$, then
$$
\overline{xy}= \overline{yx} = \{x,y\}\cup\{z\in V: [xzy]\}
$$ 
\label{LemmaLineDiam}
\end{lema}

We now introduce some useful notation. For a set $U$ and $x\in U$, we denote $U\setminus\{x\}$ by $U-x$, and for $y\notin U$, we denote $U\cup\{y\}$ by $U+y$. Consider a graph $G=(V,E)$ and $x\in V$. The neighborhood at distance $i\geq 1$ of $x$ will be denoted by $N^i(x)=\{y\in V: d(x,y)=i\}$. The neighborhood of $x$, namely $N^1(x)$, will be simply denoted by $N(x)$. The degree of the vertex $x$ is defined as $d(x)=|N(x)|$. We define the set of lines $\Li^x=\{\li{xy}:y\in V-x\}$ and for $i\geq 1$, we define $\Li^x_i=\{\li{xy}:y\in N^i(x)\}$. These last types of lines will be of special interest when $i\in \{2,3\}$ due to the following Lemma.

\begin{lema}
Let $G=(V,E)$ be a graph of diameter 3, $x\in V$ and $i\in \{2,3\}$. For distinct vertices $y,y'\in N^i(x)$, we have that $y'\notin \li{xy}$ and hence $\li{xy}\neq \li{xy'}$. In particular, for $y\in N^i(x)$, it holds that $\li{xy}\cap N^i(x)=\{y\}$. Moreover, for $A\subseteq N^i(x)$, we have that $|\{\li{xy}:y\in A\}|=|A|$.
\label{LemmaNi}
\end{lema}

The last statement of Lemma \ref{LemmaNi} will be constantly used to count lines. We now state two further lemmas that together with Lemma \ref{LemmaNi} will be especially useful in determining which vertices do not belong to a given line.

\begin{lema}
Let $G=(V,E)$ be a graph of diameter 3. Consider distinct vertices $x,y,z\in V$ such that $d(z,x)=d(z,y)\in \{2,3\}$, then $z\notin \li{xy}$.
\label{Lemmadxdy}    
\end{lema}    

\begin{lema}
Let $G=(V,E)$ be a graph of diameter 3. Let $S$ a subset of $V$. If $S$ is independent or complete, then for distinct vertices $x,y\in S$ we have that $\li{xy}\cap S=\{x,y\}$.
\label{LemmaIndCom}
\end{lema}

The proofs of these four lemmas are straightforward and are therefore omitted. These lemmas will be used repeatedly throughout our proofs, particularly to establish containment and non--containment of vertices in lines, to show that given lines are distinct, and to count lines. To illustrate their use, we include explicit references to these lemmas in the introduction of Section \ref{sec:Rnonempty}. From Section \ref{sec:R1nonempty} onward, we omit such references to avoid overloading the text with repetitive arguments.  The reader is encouraged to keep these lemmas in mind and to revisit them whenever an argument of this type is unclear.

We will now define some graphs of diameter three having less lines than vertices.
\begin{figure}[H]
    \centering
    \begin{subfigure}[t]{0.15\textwidth}
        \centering
        \resizebox{0.75\textwidth}{!}{\tikzset{every picture/.style={line width=0.75pt}} 

\begin{tikzpicture}[x=0.75pt,y=0.75pt,yscale=-1,xscale=1]

\draw  [fill={rgb, 255:red, 0; green, 0; blue, 0 }  ,fill opacity=1 ] (270,115) .. controls (270,112.24) and (272.24,110) .. (275,110) .. controls (277.76,110) and (280,112.24) .. (280,115) .. controls (280,117.76) and (277.76,120) .. (275,120) .. controls (272.24,120) and (270,117.76) .. (270,115) -- cycle ;
\draw  [fill={rgb, 255:red, 0; green, 0; blue, 0 }  ,fill opacity=1 ] (300,75) .. controls (300,72.24) and (302.24,70) .. (305,70) .. controls (307.76,70) and (310,72.24) .. (310,75) .. controls (310,77.76) and (307.76,80) .. (305,80) .. controls (302.24,80) and (300,77.76) .. (300,75) -- cycle ;
\draw  [fill={rgb, 255:red, 0; green, 0; blue, 0 }  ,fill opacity=1 ] (240,75) .. controls (240,72.24) and (242.24,70) .. (245,70) .. controls (247.76,70) and (250,72.24) .. (250,75) .. controls (250,77.76) and (247.76,80) .. (245,80) .. controls (242.24,80) and (240,77.76) .. (240,75) -- cycle ;
\draw  [fill={rgb, 255:red, 0; green, 0; blue, 0 }  ,fill opacity=1 ] (270,35) .. controls (270,32.24) and (272.24,30) .. (275,30) .. controls (277.76,30) and (280,32.24) .. (280,35) .. controls (280,37.76) and (277.76,40) .. (275,40) .. controls (272.24,40) and (270,37.76) .. (270,35) -- cycle ;
\draw    (275,35) -- (245,75) ;
\draw  [fill={rgb, 255:red, 0; green, 0; blue, 0 }  ,fill opacity=1 ] (270,165) .. controls (270,162.24) and (272.24,160) .. (275,160) .. controls (277.76,160) and (280,162.24) .. (280,165) .. controls (280,167.76) and (277.76,170) .. (275,170) .. controls (272.24,170) and (270,167.76) .. (270,165) -- cycle ;
\draw    (275,35) -- (305,75) ;
\draw    (245,75) -- (275,115) ;
\draw    (305,75) -- (275,115) ;
\draw    (275,115) -- (275,165) ;

\end{tikzpicture}}
        \caption{}
        \label{fig:a}
    \end{subfigure}
    \hspace{0.01\textwidth}
    \begin{subfigure}[t]{0.15\textwidth}
        \centering
        \resizebox{0.75\textwidth}{!}{\tikzset{every picture/.style={line width=0.75pt}} 

\begin{tikzpicture}[x=0.75pt,y=0.75pt,yscale=-1,xscale=1]

\draw  [fill={rgb, 255:red, 0; green, 0; blue, 0 }  ,fill opacity=1 ] (290,135) .. controls (290,132.24) and (292.24,130) .. (295,130) .. controls (297.76,130) and (300,132.24) .. (300,135) .. controls (300,137.76) and (297.76,140) .. (295,140) .. controls (292.24,140) and (290,137.76) .. (290,135) -- cycle ;
\draw  [fill={rgb, 255:red, 0; green, 0; blue, 0 }  ,fill opacity=1 ] (320,95) .. controls (320,92.24) and (322.24,90) .. (325,90) .. controls (327.76,90) and (330,92.24) .. (330,95) .. controls (330,97.76) and (327.76,100) .. (325,100) .. controls (322.24,100) and (320,97.76) .. (320,95) -- cycle ;
\draw  [fill={rgb, 255:red, 0; green, 0; blue, 0 }  ,fill opacity=1 ] (260,95) .. controls (260,92.24) and (262.24,90) .. (265,90) .. controls (267.76,90) and (270,92.24) .. (270,95) .. controls (270,97.76) and (267.76,100) .. (265,100) .. controls (262.24,100) and (260,97.76) .. (260,95) -- cycle ;
\draw  [fill={rgb, 255:red, 0; green, 0; blue, 0 }  ,fill opacity=1 ] (290,55) .. controls (290,52.24) and (292.24,50) .. (295,50) .. controls (297.76,50) and (300,52.24) .. (300,55) .. controls (300,57.76) and (297.76,60) .. (295,60) .. controls (292.24,60) and (290,57.76) .. (290,55) -- cycle ;
\draw    (295,55) -- (265,95) ;
\draw  [fill={rgb, 255:red, 0; green, 0; blue, 0 }  ,fill opacity=1 ] (290,185) .. controls (290,182.24) and (292.24,180) .. (295,180) .. controls (297.76,180) and (300,182.24) .. (300,185) .. controls (300,187.76) and (297.76,190) .. (295,190) .. controls (292.24,190) and (290,187.76) .. (290,185) -- cycle ;
\draw    (295,55) -- (325,95) ;
\draw    (265,95) -- (295,135) ;
\draw    (325,95) -- (295,135) ;
\draw    (295,135) -- (295,185) ;
\draw    (325,95) -- (265,95) ;

\end{tikzpicture}}
        \caption{}
        \label{fig:b}
    \end{subfigure}
    \hspace{0.01\textwidth}
    \begin{subfigure}[t]{0.15\textwidth}
        \centering
        \resizebox{0.75\textwidth}{!}{\tikzset{every picture/.style={line width=0.75pt}} 

\begin{tikzpicture}[x=0.75pt,y=0.75pt,yscale=-1,xscale=1]

\draw  [fill={rgb, 255:red, 0; green, 0; blue, 0 }  ,fill opacity=1 ] (280,165) .. controls (280,162.24) and (282.24,160) .. (285,160) .. controls (287.76,160) and (290,162.24) .. (290,165) .. controls (290,167.76) and (287.76,170) .. (285,170) .. controls (282.24,170) and (280,167.76) .. (280,165) -- cycle ;
\draw  [fill={rgb, 255:red, 0; green, 0; blue, 0 }  ,fill opacity=1 ] (340,115) .. controls (340,112.24) and (342.24,110) .. (345,110) .. controls (347.76,110) and (350,112.24) .. (350,115) .. controls (350,117.76) and (347.76,120) .. (345,120) .. controls (342.24,120) and (340,117.76) .. (340,115) -- cycle ;
\draw  [fill={rgb, 255:red, 0; green, 0; blue, 0 }  ,fill opacity=1 ] (280,115) .. controls (280,112.24) and (282.24,110) .. (285,110) .. controls (287.76,110) and (290,112.24) .. (290,115) .. controls (290,117.76) and (287.76,120) .. (285,120) .. controls (282.24,120) and (280,117.76) .. (280,115) -- cycle ;
\draw  [fill={rgb, 255:red, 0; green, 0; blue, 0 }  ,fill opacity=1 ] (310,75) .. controls (310,72.24) and (312.24,70) .. (315,70) .. controls (317.76,70) and (320,72.24) .. (320,75) .. controls (320,77.76) and (317.76,80) .. (315,80) .. controls (312.24,80) and (310,77.76) .. (310,75) -- cycle ;
\draw    (315,75) -- (285,115) ;
\draw    (315,75) -- (345,115) ;
\draw    (285,115) -- (285,165) ;
\draw    (345,115) -- (285,115) ;
\draw  [fill={rgb, 255:red, 0; green, 0; blue, 0 }  ,fill opacity=1 ] (280,215) .. controls (280,212.24) and (282.24,210) .. (285,210) .. controls (287.76,210) and (290,212.24) .. (290,215) .. controls (290,217.76) and (287.76,220) .. (285,220) .. controls (282.24,220) and (280,217.76) .. (280,215) -- cycle ;
\draw    (285,165) -- (285,215) ;

\end{tikzpicture}}
        \caption{}
        \label{fig:c}
    \end{subfigure}
    \hspace{0.01\textwidth}
    \begin{subfigure}[t]{0.15\textwidth}
        \centering
        \resizebox{0.75\textwidth}{!}{\tikzset{every picture/.style={line width=0.75pt}} 

\begin{tikzpicture}[x=0.75pt,y=0.75pt,yscale=-1,xscale=1]

\draw  [fill={rgb, 255:red, 0; green, 0; blue, 0 }  ,fill opacity=1 ] (250,195) .. controls (250,192.24) and (252.24,190) .. (255,190) .. controls (257.76,190) and (260,192.24) .. (260,195) .. controls (260,197.76) and (257.76,200) .. (255,200) .. controls (252.24,200) and (250,197.76) .. (250,195) -- cycle ;
\draw  [fill={rgb, 255:red, 0; green, 0; blue, 0 }  ,fill opacity=1 ] (310,145) .. controls (310,142.24) and (312.24,140) .. (315,140) .. controls (317.76,140) and (320,142.24) .. (320,145) .. controls (320,147.76) and (317.76,150) .. (315,150) .. controls (312.24,150) and (310,147.76) .. (310,145) -- cycle ;
\draw  [fill={rgb, 255:red, 0; green, 0; blue, 0 }  ,fill opacity=1 ] (250,145) .. controls (250,142.24) and (252.24,140) .. (255,140) .. controls (257.76,140) and (260,142.24) .. (260,145) .. controls (260,147.76) and (257.76,150) .. (255,150) .. controls (252.24,150) and (250,147.76) .. (250,145) -- cycle ;
\draw  [fill={rgb, 255:red, 0; green, 0; blue, 0 }  ,fill opacity=1 ] (280,105) .. controls (280,102.24) and (282.24,100) .. (285,100) .. controls (287.76,100) and (290,102.24) .. (290,105) .. controls (290,107.76) and (287.76,110) .. (285,110) .. controls (282.24,110) and (280,107.76) .. (280,105) -- cycle ;
\draw    (285,105) -- (255,145) ;
\draw    (285,105) -- (315,145) ;
\draw    (255,145) -- (255,195) ;
\draw  [fill={rgb, 255:red, 0; green, 0; blue, 0 }  ,fill opacity=1 ] (280,55) .. controls (280,52.24) and (282.24,50) .. (285,50) .. controls (287.76,50) and (290,52.24) .. (290,55) .. controls (290,57.76) and (287.76,60) .. (285,60) .. controls (282.24,60) and (280,57.76) .. (280,55) -- cycle ;
\draw    (285,55) -- (285,105) ;

\end{tikzpicture}}
        \caption{}
        \label{fig:d}
    \end{subfigure}
    \caption{}
    \label{fig:fewlinesA10}
\end{figure}
Let $G_a,G_b,G_c$ and $G_d$ be the graphs $(a),(b),(c)$ and $(d)$ from Figure \ref{fig:fewlinesA10}, respectively. These four graphs have diameter three and four lines, hence they all have less lines than vertices. These facts can be easily verified by the reader. Given $p, p_1,...,p_q\in  \mathbb{N}\setminus\{0\}$ such that $\sum_{i=1}^qp_i\leq p$, we define the graph $M_{p,p_1,...,p_q}$ given by joining a complete graph $K_p$ with the complete graphs $K_{p_1},...,K_{p_q}$ by a matching $M\subseteq E(K_p,\bigcup_{i=1}^q K_{p_i})$ of size $\sum_{i=1}^qp_i$. We also define the graph $M'_{2p}$ given by joining two complete graphs $K_p$ by a matching of size $p-1$.  In the following two propositions, we count the number of lines of graphs in the two latter families. Furthermore, we make explicit which graphs belonging to them have less lines than vertices. In a graph belonging to any of the two families, the line generated by an edge in the matching is universal, therefore both families satisfy the Chen--Chvátal conjecture. The proofs of Propositions \ref{PropFam1} and \ref{PropFam2} are routine and are thus omitted.

\begin{prop}
Given $p\geq 3$ and  $p_1,...,p_q\in  \mathbb{N}\setminus\{0\}$ with $q\geq 2$ such that $\sum_{i=1}^qp_i\leq p$, the graph $M_{p,p_1,...,p_q}$ has diameter 3 and 
$$
\ell(M_{p,p_1,...,p_q})=\binom{p}{2}+1
$$
In particular, 

{\small 
$$
\ell(M_{p,p_1,...,p_q})<|V(M_{p,p_1,...,p_q})| \Longleftrightarrow M_{p,p_1,...,p_q}\in \{M_{3,1,1}, M_{3,1,1,1}, M_{3,2,1}, M_{4,1,1,1,1}, M_{4,2,1,1}, M_{4,2,2},M_{4,3,1}\}
$$}
\label{PropFam1}
\end{prop}

\begin{prop}
Given $p\geq 3$, the graph $M'_{2p}$ has diameter 3 and 
$$
\ell(M'_{2p})=\binom{p}{2}+1
$$
In particular,
$$
\ell(M'_{2p})<|V(M'_{2p})|\Longleftrightarrow M'_{2p}\in \{M'_6,M'_8\}
$$
\label{PropFam2}
\end{prop}

We now define the family $\mathcal{F}$, consisting of 14 graphs.

$$
\mathcal{F}=\{G_a,G_b,G_c,G_d,M_{2,1,1}, M_{3,1,1}, M_{3,1,1,1}, M_{3,2,1}, M_{4,1,1,1,1}, M_{4,2,1,1}, M_{4,2,2},M_{4,3,1},M'_6,M'_8\}
$$

$M_{2,1,1}$ is the path of length three, which has only one line, therefore all graphs in $\mathcal{F}$ have diameter three and less lines than vertices. Hence, we have established the backward implication of the following theorem.

\begin{teo}
Let $G=(V,E)$ be a graph of diameter 3. Then, $\ell(G)<|V|$ if and only if $G$ is isomorphic to one of the graphs in $\mathcal{F}$.   
\label{MainTheorem}
\end{teo}

Since all graphs in $\mathcal{F}$ have a universal line, it follows that all graphs of diameter three satisfy the Chen--Chvátal conjecture. In what remains of this paper, we will be mainly devoted to proving the forward implication of Theorem \ref{MainTheorem}. For this purpose, we will consider a graph $G$ with diameter three and less lines than vertices and we shall prove that it is isomorphic to one of the graphs in $\mathcal{F}$. The corollary of the following proposition will give us some insight into how to approach the problem. 

\begin{prop}
Let $G=(V,E)$ be a graph and $x\in V$. If $|\Li^x|\geq 2$, then $\ell(G)\geq |\Li^x|+1$.
\end{prop}

\begin{proof}
Consider a graph $G=(V,E)$ and $x\in V$ such that $|\Li^x|\geq 2$. Since all lines in $\Li^x$ contain $x$, it suffices to find a line that does not contain $x$. Note that there must exist a line in $\Li^x$ different from $V$. Let $y\in V-x$ such that $\li{xy}\neq V$ and consider $z\notin \li{xy}$. It follows that $x\notin \li{yz}$, hence $\ell(G)\geq |\Li^x|+1$.    
\end{proof}

\begin{cor}
Let $G=(V,E)$ be a graph such that $|V|\geq 3$ and $\ell(G)<|V|$.  Then $\forall \,x\in V$, $|\Li^x|<|V|-1$.  
\end{cor}

The fact that for any vertex $x$ in a graph $G=(V,E)$ of diameter three and less lines than vertices, it holds that $|\Li^x|<|V|-1$,  motivates the study of the types of line repetitions in $\Li^x$ for some $x\in V$. That is, which structures appear in the graph when $\li{xy}=\li{xz}$, for different vertices $y,z\in V-x$.  In this spirit, the proof of the main theorem will consider the case where there exists a vertex $x\in V$ such that $\Li^x_2\cap \Li^x_3\neq \emptyset$ and the case where for every vertex $x\in V$, it holds that $\Li^x_2\cap \Li^x_3=\emptyset$.
 
For the rest of this paper, let $G=(V,E)$ be a graph of diameter 3 with less lines than vertices. We denote $n=|V|$, that is, $\ell(G)\leq n-1$. We define 
$$R=\{x\in V: \Li^x_2\cap \Li^x_3\neq \emptyset\} \textrm{  and }R_1=\{x\in R:d(x)=1\}.$$

For $x\in V$, we also define $A_1(x)$ as the following set.
$$A_1(x)=\{w\in N^3(x): \li{xw}\in \Li^x_2\cap \Li^w_2\}.$$

We observe that the graphs $G_a$, $G_b$, $G_c$ and $G_d$ satisfy that $A_1(x)=\emptyset$ for every vertex $x$. In Section \ref{sec:Rnonempty} we will prove that if $R\neq\emptyset$, then $G$ is isomorphic to one of the following graphs: $M_{2,1,1}$, $M_{3,1,1}$, $M_{3,1,1,1}$, $M_{3,2,1}$, $M_{4,1,1,1,1}$, $M_{4,2,1,1}$, $M_{4,3,1}$, $M_{4,2,2}$ if there exists $x\in V$ with $A_1(x)\neq \emptyset$, or to one of the following graphs: $G_a$, $G_b$, $G_c$, $G_d$ otherwise. In Section \ref{sec:Rempty} we will prove that if $R=\emptyset$, then $G$ is isomorphic to $M'_{6}$ or $M'_{8}$. The proofs of these results mainly consist in identifying useful types of line repetitions, counting a substantial number of lines, and gradually deriving structural properties of $G$. In the case where $R=\emptyset$, we are able to find a complete description of $G$ by exploiting a symmetry found in the graph. Given a vertex $x$, we will use the letters $u$, $v$ and $w$ to denote vertices in $N(x), N^2(x)$ and $N^3(x)$, respectively.

\section{Proof of the main theorem: Case 1}\label{sec:Rnonempty}

In this section, we assume that  $R$ is nonempty. We will first define, for a vertex $x$, sets that capture two types of line repetitions in $\Li^x_2\cap \Li^x_3$. For $x\in V$, we define $A(x)=\{w\in N^3(x):\li{xw}\in \Li^x_2\}$ and $B(x)=\{v\in N^2(x):\li{xv}\in \Li ^x_3\}$. We also define $A_1(x)=\{w\in A(x): \exists \, u\in N(x) \text{ such that }\li{wu}=\li{xw}\}$ and $A_2(x)=A(x)\setminus A_1(x)$. Observe that the definition of $A_1(x)$ coincides with the one given in Section \ref{sec:Prelim}. 
We also note that $R=\{x\in V: A(x)\neq \emptyset\}$.

\begin{claimsec}
For $x\in V$, the function $f: A(x)\to B(x)$ given by $f(w)=v\in N^2(x)$ such that $\li{xv}=\li{xw}$ is well defined and bijective. Furthermore, $E(A(x),N^2(x))=E(B(x),N^3(x))=\{wf(w):w\in A(x)\}$.\label{ClaimRnonemptyfunc}
\end{claimsec}

\begin{proof} 
Fix $x\in V$ and let $w\in A(x)$. By definition there exists $v\in N^2(x)$ such that $\li{xv}=\li{xw}$. It is immediate that $v\in B(x)$ and by Lemma \ref{LemmaNi}, that this $v$ is unique, hence $f$ is well defined. The bijectivity follows directly from Lemma \ref{LemmaNi}. We now prove that $E(A(x),N^2(x))=E(B(x),N^3(x))=\{wf(w):w\in A(x)\}$, noting that for $w\in A(x)$, it holds that $f(w)\in \li{xw}\cap N^2(x)$, thus by Lemma \ref{LemmaLineDiam} we obtain that $wf(w)\in E$. By Lemma \ref{LemmaNi}, for $w\in A(x)$ and $v\in B(x)$ we have that $\left (N(w)\cap N^2(x) \right)\setminus\{f(w)\}\subseteq \li{xw}\setminus \li{xf(w)}=\emptyset$ and  $\left ( N(v)\cap N^3(x) \right) \setminus \{f^{-1}(v)\}\subseteq \li{xv}\setminus \li{xf^{-1}(v)}=\emptyset$, hence the desired equalities hold. 
\end{proof}

For $x\in V$, we define $B_1(x)=f(A_1(x))$ and $B_2(x)=f(A_2(x))$, observing that $B(x)=B_1(x)\cup B_2(x)$. Let $x\in R$. Claim \ref{ClaimRnonemptyfunc} implies that for $w\in A(x)$ and $u\in N^2(w)\cap N(x)$, it holds that  $f(w)\in N(u)\cap \li{wu}$.
We now consider the sets of lines $\Li^x_2,\Li^x_3,L_{A_2}=\{\li{wu}:w\in A_2(x),\, u\in N^2(w)\cap N(x)\}$, and we note that all the lines in $L_{A_2}$ contain $x$. Although the function $f$ from Claim \ref{ClaimRnonemptyfunc} and the set of lines $L_{A_2}$ depend on $x$, we will omit this dependence in the notation, as it will be clear from the context.

\begin{claimsec}
For $x\in R$,  $L_{A_2}$ is disjoint from $\Li^x_2$ and $\Li^x_3$. Moreover, $|L_{A_2}|=\sum_{w\in A_2(x)}|N^2(w)\cap N(x)|\geq |A_2(x)|$.
\end{claimsec}

\begin{proof}
Let $x\in R$ and $l=\li{wu}\in L_{A_2}$, i.e., $w\in A_2(x)$ and $u\in N^2(w)\cap N(x)$. We have that $w,f(w)\in l$, hence, by Lemma \ref{LemmaNi}, $l$ can only be equal to $\li{xf(w)}$ in $\Li^x_2$ and to $\li{xw}$ in $\Li^x_3$. By definition of $A_2(x)$, $l\neq \li{xw}=\li{xf(w)}$, therefore it holds that $l\notin \Li^x_2\cup \Li^x_3$, and we conclude that $L_{A_2}$ is disjoint from $\Li^x_2$ and $\Li^x_3$. By Lemma \ref{LemmaNi}, for $w\in A_2(x)$ it holds that $|\{\li{wu}:u\in N^2(w)\cap N(x)\}|=|N^2(w)\cap N(x)|$. Therefore, in order to prove that $|L_{A_2}|=\sum_{w\in A_2(x)}|N^2(w)\cap N(x)|$ it suffices to show that for distinct $w,w'\in A_2(x)$, $u\in N^2(w)\cap N(x)$ and $u'\in N^2(w')\cap N(x)$, it holds that $\li{wu}\neq \li{w'u'}$. This fact is straightforward when $d(w',u)=2$, as in this case $w'\notin \li{wu}$ by Lemma \ref{LemmaNi}. Thus we consider the case where $d(w',u)=3$, observing that $u\neq u'$. Let $v'=f(w')$, and note that $w'v'\in E$ and $wv'\notin E$ by Claim \ref{ClaimRnonemptyfunc}. As $d(w',u)=3$, it follows that $v'u\notin E$ and therefore $v'\notin  \li{wu}$ by Lemma \ref{LemmaIndCom}. As $v'\in \li{w'u'}$, we obtain that $\li{wu}\neq \li{w'u'}$. 
\end{proof}

\subsection{\texorpdfstring{$R_1$}{R1} is nonempty}\label{sec:R1nonempty}

In this section we assume that $R_1\neq \emptyset$. Let $x\in R_1$ and $N(x)=\{u\}$. We note that $\li{xu}=V$, $N^3(x)\subseteq N^2(u)$ and $N^2(x)\subseteq N(u)$. Consider $w^*\in A(x)$ and $v^*=f(w^*)$, that is, $\li{xv^*}=\li{xw^*}$. We illustrate $G$ in the following diagram.
\begin{figure}[H]
\begin{center}      
\scalebox{.75}{\tikzset{every picture/.style={line width=0.75pt}} 

\begin{tikzpicture}[x=0.75pt,y=0.75pt,yscale=-1,xscale=1]

\draw  [fill={rgb, 255:red, 0; green, 0; blue, 0 }  ,fill opacity=1 ] (215,25) .. controls (215,22.24) and (217.24,20) .. (220,20) .. controls (222.76,20) and (225,22.24) .. (225,25) .. controls (225,27.76) and (222.76,30) .. (220,30) .. controls (217.24,30) and (215,27.76) .. (215,25) -- cycle ;
\draw  [fill={rgb, 255:red, 0; green, 0; blue, 0 }  ,fill opacity=1 ] (215,85) .. controls (215,82.24) and (217.24,80) .. (220,80) .. controls (222.76,80) and (225,82.24) .. (225,85) .. controls (225,87.76) and (222.76,90) .. (220,90) .. controls (217.24,90) and (215,87.76) .. (215,85) -- cycle ;
\draw    (220,25) -- (220,85) ;
\draw  [fill={rgb, 255:red, 0; green, 0; blue, 0 }  ,fill opacity=1 ] (280,145) .. controls (280,142.24) and (282.24,140) .. (285,140) .. controls (287.76,140) and (290,142.24) .. (290,145) .. controls (290,147.76) and (287.76,150) .. (285,150) .. controls (282.24,150) and (280,147.76) .. (280,145) -- cycle ;
\draw  [fill={rgb, 255:red, 0; green, 0; blue, 0 }  ,fill opacity=1 ] (251.29,145.25) .. controls (251.29,144.59) and (250.76,144.05) .. (250.09,144.05) .. controls (249.43,144.05) and (248.89,144.59) .. (248.89,145.25) .. controls (248.89,145.91) and (249.43,146.45) .. (250.09,146.45) .. controls (250.76,146.45) and (251.29,145.91) .. (251.29,145.25) -- cycle ;
\draw  [fill={rgb, 255:red, 0; green, 0; blue, 0 }  ,fill opacity=1 ] (261.42,145.2) .. controls (261.42,144.54) and (260.88,144) .. (260.22,144) .. controls (259.55,144) and (259.02,144.54) .. (259.02,145.2) .. controls (259.02,145.86) and (259.55,146.4) .. (260.22,146.4) .. controls (260.88,146.4) and (261.42,145.86) .. (261.42,145.2) -- cycle ;
\draw  [fill={rgb, 255:red, 0; green, 0; blue, 0 }  ,fill opacity=1 ] (271.43,145.25) .. controls (271.43,144.59) and (270.89,144.05) .. (270.23,144.05) .. controls (269.57,144.05) and (269.03,144.59) .. (269.03,145.25) .. controls (269.03,145.91) and (269.57,146.45) .. (270.23,146.45) .. controls (270.89,146.45) and (271.43,145.91) .. (271.43,145.25) -- cycle ;
\draw  [fill={rgb, 255:red, 0; green, 0; blue, 0 }  ,fill opacity=1 ] (140,145) .. controls (140,142.24) and (142.24,140) .. (145,140) .. controls (147.76,140) and (150,142.24) .. (150,145) .. controls (150,147.76) and (147.76,150) .. (145,150) .. controls (142.24,150) and (140,147.76) .. (140,145) -- cycle ;
\draw  [fill={rgb, 255:red, 0; green, 0; blue, 0 }  ,fill opacity=1 ] (180,145) .. controls (180,142.24) and (182.24,140) .. (185,140) .. controls (187.76,140) and (190,142.24) .. (190,145) .. controls (190,147.76) and (187.76,150) .. (185,150) .. controls (182.24,150) and (180,147.76) .. (180,145) -- cycle ;
\draw    (220,85) -- (145,145) ;
\draw    (220,85) -- (185,145) ;
\draw    (220,85) -- (235,145) ;
\draw  [fill={rgb, 255:red, 0; green, 0; blue, 0 }  ,fill opacity=1 ] (140,205) .. controls (140,202.24) and (142.24,200) .. (145,200) .. controls (147.76,200) and (150,202.24) .. (150,205) .. controls (150,207.76) and (147.76,210) .. (145,210) .. controls (142.24,210) and (140,207.76) .. (140,205) -- cycle ;
\draw  [fill={rgb, 255:red, 0; green, 0; blue, 0 }  ,fill opacity=1 ] (180,205) .. controls (180,202.24) and (182.24,200) .. (185,200) .. controls (187.76,200) and (190,202.24) .. (190,205) .. controls (190,207.76) and (187.76,210) .. (185,210) .. controls (182.24,210) and (180,207.76) .. (180,205) -- cycle ;
\draw    (145,145) -- (145,205) ;
\draw    (185,145) -- (185,205) ;
\draw  [dash pattern={on 4.5pt off 4.5pt}]  (185,145) -- (145,205) ;
\draw  [dash pattern={on 4.5pt off 4.5pt}]  (235,145) -- (145,205) ;
\draw   (139.9,213.08) .. controls (139.91,217.75) and (142.24,220.08) .. (146.91,220.07) -- (179.96,220.01) .. controls (186.63,220) and (189.97,222.33) .. (189.98,227) .. controls (189.97,222.33) and (193.29,219.99) .. (199.96,219.98)(196.96,219.98) -- (233.01,219.92) .. controls (237.68,219.91) and (240.01,217.58) .. (240,212.91) ;
\draw    (235,145) -- (235,205) ;
\draw  [fill={rgb, 255:red, 0; green, 0; blue, 0 }  ,fill opacity=1 ] (230,205) .. controls (230,202.24) and (232.24,200) .. (235,200) .. controls (237.76,200) and (240,202.24) .. (240,205) .. controls (240,207.76) and (237.76,210) .. (235,210) .. controls (232.24,210) and (230,207.76) .. (230,205) -- cycle ;
\draw  [fill={rgb, 255:red, 0; green, 0; blue, 0 }  ,fill opacity=1 ] (230,145) .. controls (230,142.24) and (232.24,140) .. (235,140) .. controls (237.76,140) and (240,142.24) .. (240,145) .. controls (240,147.76) and (237.76,150) .. (235,150) .. controls (232.24,150) and (230,147.76) .. (230,145) -- cycle ;
\draw  [fill={rgb, 255:red, 0; green, 0; blue, 0 }  ,fill opacity=1 ] (201.29,145.25) .. controls (201.29,144.59) and (200.76,144.05) .. (200.09,144.05) .. controls (199.43,144.05) and (198.89,144.59) .. (198.89,145.25) .. controls (198.89,145.91) and (199.43,146.45) .. (200.09,146.45) .. controls (200.76,146.45) and (201.29,145.91) .. (201.29,145.25) -- cycle ;
\draw  [fill={rgb, 255:red, 0; green, 0; blue, 0 }  ,fill opacity=1 ] (211.42,145.2) .. controls (211.42,144.54) and (210.88,144) .. (210.22,144) .. controls (209.55,144) and (209.02,144.54) .. (209.02,145.2) .. controls (209.02,145.86) and (209.55,146.4) .. (210.22,146.4) .. controls (210.88,146.4) and (211.42,145.86) .. (211.42,145.2) -- cycle ;
\draw  [fill={rgb, 255:red, 0; green, 0; blue, 0 }  ,fill opacity=1 ] (221.43,145.25) .. controls (221.43,144.59) and (220.89,144.05) .. (220.23,144.05) .. controls (219.57,144.05) and (219.03,144.59) .. (219.03,145.25) .. controls (219.03,145.91) and (219.57,146.45) .. (220.23,146.45) .. controls (220.89,146.45) and (221.43,145.91) .. (221.43,145.25) -- cycle ;
\draw    (220,85) -- (285,145) ;
\draw  [dash pattern={on 4.5pt off 4.5pt}]  (285,145) -- (145,205) ;
\draw  [fill={rgb, 255:red, 0; green, 0; blue, 0 }  ,fill opacity=1 ] (272.4,205.25) .. controls (272.4,204.59) and (271.86,204.05) .. (271.2,204.05) .. controls (270.54,204.05) and (270,204.59) .. (270,205.25) .. controls (270,205.91) and (270.54,206.45) .. (271.2,206.45) .. controls (271.86,206.45) and (272.4,205.91) .. (272.4,205.25) -- cycle ;
\draw  [fill={rgb, 255:red, 0; green, 0; blue, 0 }  ,fill opacity=1 ] (282.52,205.2) .. controls (282.52,204.54) and (281.98,204) .. (281.32,204) .. controls (280.66,204) and (280.12,204.54) .. (280.12,205.2) .. controls (280.12,205.86) and (280.66,206.4) .. (281.32,206.4) .. controls (281.98,206.4) and (282.52,205.86) .. (282.52,205.2) -- cycle ;
\draw  [fill={rgb, 255:red, 0; green, 0; blue, 0 }  ,fill opacity=1 ] (292.54,205.25) .. controls (292.54,204.59) and (292,204.05) .. (291.34,204.05) .. controls (290.68,204.05) and (290.14,204.59) .. (290.14,205.25) .. controls (290.14,205.91) and (290.68,206.45) .. (291.34,206.45) .. controls (292,206.45) and (292.54,205.91) .. (292.54,205.25) -- cycle ;

\draw (216,2.4) node [anchor=north west][inner sep=0.75pt]    {$x$};
\draw (54,82.4) node [anchor=north west][inner sep=0.75pt]    {$N( x)$};
\draw (54,136.4) node [anchor=north west][inner sep=0.75pt]    {$N^{2}( x)$};
\draw (54,196.4) node [anchor=north west][inner sep=0.75pt]    {$N^{3}( x)$};
\draw (228.33,73.73) node [anchor=north west][inner sep=0.75pt]    {$u$};
\draw (116.67,138.15) node [anchor=north west][inner sep=0.75pt]    {$v^{*}$};
\draw (114.67,197.4) node [anchor=north west][inner sep=0.75pt]    {$w^{*}$};
\draw (175.04,231.4) node [anchor=north west][inner sep=0.75pt]    {$A( x)$};

\end{tikzpicture}}
\end{center}
\caption{Diagram of $G$. Dashed lines represent missing edges.}    
\label{fig:R1nonemptyDiagram}
\end{figure}
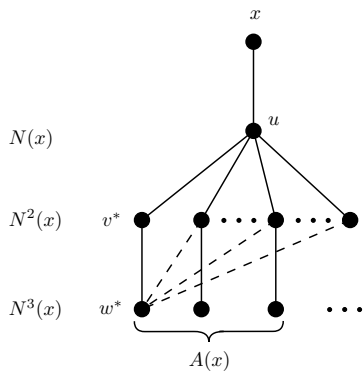

We now define the set of lines $L^*=\{\li{w^*v}:v\in N^2(x)-v^*\}$. For a line $l=\li{w^*v}$ with $v\in N^2(x)-v^*$, we have that $u\in l$ if $d(w^*,v)=3$ and $u\notin l$ if $d(w^*,v)=2$, hence $|L^*|=|N^2(x)|-1$. The lines in $L^*$ do not contain $x$ and thus $L^*$ is disjoint from $\Li^x_2$, $\Li^x_3$ and $L_{A_2}$. The set $\Li^x_2\cup \Li^x_3\cup L_{A_2}\cup L^*$ contains 
{\small $$
|N^2(x)|+ |N^3(x)|- |A(x)|+ |A_2(x)|+ |N^2(x)|-1=|N^2(x)|+ |N^3(x)|+ |N^2(x)\setminus B_1(x)|-1=n-3 +|N^2(x)\setminus B_1(x)|$$}
lines. We will prove the following proposition.

\begin{prop}
If $A_1(x)\neq \emptyset$, then $G$ is isomorphic to one of the following graphs: $M_{2,1,1}$, $M_{3,1,1}$, $M_{3,1,1,1}$, $M_{3,2,1}$, $M_{4,1,1,1,1}$, $M_{4,2,1,1}$, $M_{4,3,1}$.
\label{PropR1A_1nonempty}
\end{prop}

\subsubsection*{Proof of Proposition \ref{PropR1A_1nonempty}}
We assume that $A_1(x)\neq \emptyset$.
If $|N^2(x)|=1$, necessarily $|N^3(x)|=1$, and hence $G$ is isomorphic to $M_{2,1,1}$. In what follows we assume that $|N^2(x)|\geq 2$. For $w\in A_1(x)$ and $v=f(w)$, we have that $\li{xv}=\li{xw}=\li{wu}$, and since $\li{xv}\cap N^2(x)=\{v\}$ and $N^3(w)\cap N^2(x)\subseteq \li{wu}\setminus \li{xv}=\emptyset$, we obtain that $N^2(x)-v  \subseteq N^2(w)$. Without loss of generality, we may assume that $w^*\in A_1(x)$,  therefore,  $N^2(x)-v^*  \subseteq N^2(w^*)$. Since $|N^2(x)|\geq 2$,  $V\notin \Li^x_2$, in particular, $\li{xv^*}=\li{xw^*}\neq V$. 
As the lines in $\Li^x_3\setminus \{\li{xw^*}\}$ do not contain $w^*$, we also obtain that $V\notin \Li^x_3$.
Given that $d(w^*,u)=2$, no line in $L_{A_2}$ contains $w^*$, and recalling that the lines in $L^*$ do not contain $x$, we conclude that $V\notin \Li^x_2\cup \Li^x_3\cup  L_{A_2}\cup L^*$.
As $V=\li{xu}$, the set $\Li^x_2\cup \Li^x_3\cup L_{A_2}\cup L^*\cup \{\li{xu}\}$ contains $n-2 +|N^2(x)\setminus B_1(x)|$ lines.

\begin{claim}
$G[N^2(x)]$ is a complete graph.
\end{claim}

\begin{proof}
For the sake of contradiction, suppose that there exist distinct $v,v'\in N^2(x)$ such that $vv'\notin E$, in particular $d(v,v')=2$. We consider the line $l_1=\li{vv'}$, and we note that $x\notin l_1$. Since $\{v,v'\}\cap N^2(w^*)\neq \emptyset$, we also have that $w^*\notin l_1$, implying that $l_1\notin \Li^x_2\cup \Li^x_3\cup L_{A_2}\cup L^*\cup \{\li{xu}\}$.  It follows that $\ell(G)=n-1$, $\Li(G)= \Li^x_2\cup \Li^x_3\cup L_{A_2}\cup L^*\cup \{\li{xu},l_1\}$, $|N^2(x)\setminus B_1(x)|=\emptyset$ and $v,v'\in B_1(x)$. Consider $w=f^{-1}(v)$ and $w'=f^{-1}(v')$. Without loss of generality, we assume that $w^*=w'$ and $v^*=v'$, therefore $d(v^*,v)=2$ and $l_1=\li{v^*v}$. $G$ is depicted in the following diagram.
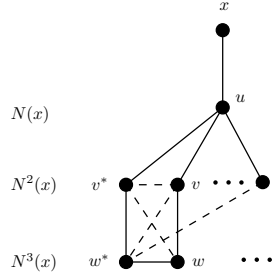
\begin{center}
\scalebox{.65}{\tikzset{every picture/.style={line width=0.75pt}} 

\begin{tikzpicture}[x=0.75pt,y=0.75pt,yscale=-1,xscale=1]

\draw  [fill={rgb, 255:red, 0; green, 0; blue, 0 }  ,fill opacity=1 ] (268.93,57) .. controls (268.93,54.24) and (271.17,52) .. (273.93,52) .. controls (276.69,52) and (278.93,54.24) .. (278.93,57) .. controls (278.93,59.76) and (276.69,62) .. (273.93,62) .. controls (271.17,62) and (268.93,59.76) .. (268.93,57) -- cycle ;
\draw  [fill={rgb, 255:red, 0; green, 0; blue, 0 }  ,fill opacity=1 ] (268.93,117) .. controls (268.93,114.24) and (271.17,112) .. (273.93,112) .. controls (276.69,112) and (278.93,114.24) .. (278.93,117) .. controls (278.93,119.76) and (276.69,122) .. (273.93,122) .. controls (271.17,122) and (268.93,119.76) .. (268.93,117) -- cycle ;
\draw    (273.93,57) -- (273.93,117) ;
\draw  [fill={rgb, 255:red, 0; green, 0; blue, 0 }  ,fill opacity=1 ] (193.93,177) .. controls (193.93,174.24) and (196.17,172) .. (198.93,172) .. controls (201.69,172) and (203.93,174.24) .. (203.93,177) .. controls (203.93,179.76) and (201.69,182) .. (198.93,182) .. controls (196.17,182) and (193.93,179.76) .. (193.93,177) -- cycle ;
\draw  [fill={rgb, 255:red, 0; green, 0; blue, 0 }  ,fill opacity=1 ] (233.93,177) .. controls (233.93,174.24) and (236.17,172) .. (238.93,172) .. controls (241.69,172) and (243.93,174.24) .. (243.93,177) .. controls (243.93,179.76) and (241.69,182) .. (238.93,182) .. controls (236.17,182) and (233.93,179.76) .. (233.93,177) -- cycle ;
\draw    (273.93,117) -- (198.93,177) ;
\draw    (273.93,117) -- (238.93,177) ;
\draw    (273.93,117) -- (305,175) ;
\draw  [fill={rgb, 255:red, 0; green, 0; blue, 0 }  ,fill opacity=1 ] (193.93,237) .. controls (193.93,234.24) and (196.17,232) .. (198.93,232) .. controls (201.69,232) and (203.93,234.24) .. (203.93,237) .. controls (203.93,239.76) and (201.69,242) .. (198.93,242) .. controls (196.17,242) and (193.93,239.76) .. (193.93,237) -- cycle ;
\draw  [fill={rgb, 255:red, 0; green, 0; blue, 0 }  ,fill opacity=1 ] (233.93,237) .. controls (233.93,234.24) and (236.17,232) .. (238.93,232) .. controls (241.69,232) and (243.93,234.24) .. (243.93,237) .. controls (243.93,239.76) and (241.69,242) .. (238.93,242) .. controls (236.17,242) and (233.93,239.76) .. (233.93,237) -- cycle ;
\draw    (198.93,177) -- (198.93,237) ;
\draw    (238.93,177) -- (238.93,237) ;
\draw  [dash pattern={on 4.5pt off 4.5pt}]  (238.93,177) -- (198.93,237) ;
\draw  [dash pattern={on 4.5pt off 4.5pt}]  (305,175) -- (198.93,237) ;
\draw  [fill={rgb, 255:red, 0; green, 0; blue, 0 }  ,fill opacity=1 ] (300,175) .. controls (300,172.24) and (302.24,170) .. (305,170) .. controls (307.76,170) and (310,172.24) .. (310,175) .. controls (310,177.76) and (307.76,180) .. (305,180) .. controls (302.24,180) and (300,177.76) .. (300,175) -- cycle ;
\draw  [fill={rgb, 255:red, 0; green, 0; blue, 0 }  ,fill opacity=1 ] (269.54,175.25) .. controls (269.54,174.59) and (269.01,174.05) .. (268.34,174.05) .. controls (267.68,174.05) and (267.14,174.59) .. (267.14,175.25) .. controls (267.14,175.91) and (267.68,176.45) .. (268.34,176.45) .. controls (269.01,176.45) and (269.54,175.91) .. (269.54,175.25) -- cycle ;
\draw  [fill={rgb, 255:red, 0; green, 0; blue, 0 }  ,fill opacity=1 ] (279.67,175.2) .. controls (279.67,174.54) and (279.13,174) .. (278.47,174) .. controls (277.8,174) and (277.27,174.54) .. (277.27,175.2) .. controls (277.27,175.86) and (277.8,176.4) .. (278.47,176.4) .. controls (279.13,176.4) and (279.67,175.86) .. (279.67,175.2) -- cycle ;
\draw  [fill={rgb, 255:red, 0; green, 0; blue, 0 }  ,fill opacity=1 ] (289.68,175.25) .. controls (289.68,174.59) and (289.14,174.05) .. (288.48,174.05) .. controls (287.82,174.05) and (287.28,174.59) .. (287.28,175.25) .. controls (287.28,175.91) and (287.82,176.45) .. (288.48,176.45) .. controls (289.14,176.45) and (289.68,175.91) .. (289.68,175.25) -- cycle ;
\draw  [dash pattern={on 4.5pt off 4.5pt}]  (238.93,177) -- (198.93,177) ;
\draw    (198.93,237) -- (238.93,237) ;
\draw  [dash pattern={on 4.5pt off 4.5pt}]  (198.93,177) -- (238.93,237) ;
\draw  [fill={rgb, 255:red, 0; green, 0; blue, 0 }  ,fill opacity=1 ] (291.46,235.25) .. controls (291.46,234.59) and (290.92,234.05) .. (290.26,234.05) .. controls (289.6,234.05) and (289.06,234.59) .. (289.06,235.25) .. controls (289.06,235.91) and (289.6,236.45) .. (290.26,236.45) .. controls (290.92,236.45) and (291.46,235.91) .. (291.46,235.25) -- cycle ;
\draw  [fill={rgb, 255:red, 0; green, 0; blue, 0 }  ,fill opacity=1 ] (301.58,235.2) .. controls (301.58,234.54) and (301.05,234) .. (300.38,234) .. controls (299.72,234) and (299.18,234.54) .. (299.18,235.2) .. controls (299.18,235.86) and (299.72,236.4) .. (300.38,236.4) .. controls (301.05,236.4) and (301.58,235.86) .. (301.58,235.2) -- cycle ;
\draw  [fill={rgb, 255:red, 0; green, 0; blue, 0 }  ,fill opacity=1 ] (311.6,235.25) .. controls (311.6,234.59) and (311.06,234.05) .. (310.4,234.05) .. controls (309.74,234.05) and (309.2,234.59) .. (309.2,235.25) .. controls (309.2,235.91) and (309.74,236.45) .. (310.4,236.45) .. controls (311.06,236.45) and (311.6,235.91) .. (311.6,235.25) -- cycle ;

\draw (269.93,34.4) node [anchor=north west][inner sep=0.75pt]    {$x$};
\draw (107.93,114.4) node [anchor=north west][inner sep=0.75pt]    {$N( x)$};
\draw (107.93,168.4) node [anchor=north west][inner sep=0.75pt]    {$N^{2}( x)$};
\draw (107.93,228.4) node [anchor=north west][inner sep=0.75pt]    {$N^{3}( x)$};
\draw (282.27,105.73) node [anchor=north west][inner sep=0.75pt]    {$u$};
\draw (170.6,170.15) node [anchor=north west][inner sep=0.75pt]    {$v^{*}$};
\draw (168.6,229.4) node [anchor=north west][inner sep=0.75pt]    {$w^{*}$};
\draw (248.33,172.96) node [anchor=north west][inner sep=0.75pt]    {$v$};
\draw (248.78,232.96) node [anchor=north west][inner sep=0.75pt]    {$w$};

\end{tikzpicture}}
\end{center}
As $d(w^*,v)=2$, necessarily $w^*w\in E$. Now consider the line $l_2=\li{w^*w}$, observing that $x,u\notin l_2$ and that $v,v^*\in l_2$. Since $u\in l_1$, we have that $l_2\in L^*$, and since $v\in l_2$, $l_2$ must be equal to $\li{w^*v}$.
Note that $v^*\in l_2\setminus \li{w^*v}$, yielding a contradiction.    
\end{proof}

Given that $G[N^2(x)]$ is a complete graph, we define the set of lines $L_2=\{\li{vv'}:v,v'\in N^2(x), v\neq v'\}$. Since no line in $L_2$ contains $x$, the set $L_2$ is disjoint from $\Li^x_2,\Li^x_3$ and $L_{A_2}$, furthermore, $V\notin L_2$. Therefore, the set $\Li^x_2\cup \Li^x_3\cup L_{A_2}\cup L_2\cup \{\li{xu}\}$ contains 

$$
|N^2(x)|+ |N^3(x)|- |A(x)|+ |A_2(x)|+ \binom{|N^2(x)|}{2}+1=n-1-|A_1(x)|+\frac{|N^2(x)|^2-|N^2(x)|}{2}
$$

lines. As $\ell(G)\leq n-1$, we obtain that $|N^2(x)|^2\leq 2|A_1(x)|+|N^2(x)|$, hence it must hold that $(|A_1(x)|,|N^2(x)|)\in \{(1,2),(2,2),(3,3)\}$. We will prove that $G$ is isomorphic to  $M_{p,p_1,...,p_q}$, with $p=|N^2(x)|+1$, since $G[N^2(x)+u]$ is a complete graph.

\begin{claim}
If $(|A_1(x)|,|N^2(x)|)=(1,2)$, then $G$ is isomorphic to $M_{3,1,1}$.
\end{claim}

\begin{proof}
In this case we have that $\ell(G)=n-1$ and $\Li(G)=\Li^x_2\cup \Li^x_3\cup L_{A_2}\cup L_2\cup \{\li{xu}\}$. Moreover, $A_1(x)=\{w^*\}$ and $N^2(x)=\{v^*,v\}$, thus it suffices to prove that $N^3(x)=\{w^*\}$. Suppose, for contradiction, that there exists $w\in N^3(x)-w^*$. Necessarily we have that  $wv\in E$, and since $w\notin A_1(x)$, the set $N^3(x)$ must be different from $\{w^*,w\}$. It follows that there exists $w'\in N^3(x)\setminus\{w^*,w\}$ and that $w'v\in E$. We note that $x,v^*\notin l=\li{ww'}$, hence $l\notin \Li(G)$, which is a contradiction.
\end{proof}

\begin{claim}
If $(|A_1(x)|,|N^2(x)|)=(2,2)$, then $G$ is isomorphic to  $M_{3,1,1,1}$ or $M_{3,2,1}$.
\end{claim}

\begin{proof}
Since $|A_1(x)|=|N^2(x)|$, it must hold that $N^3(x)=A_1(x)$ and $N^2(x)=B_1(x)$.  Let $A_1(x)=\{w^*,w\}$. If $w^*w\in E$, then $G$ is isomorphic to $M_{3,2,1}$, and if $w^*w\notin E$, then $G$ is isomorphic to $M_{3,1,1,1}$.
\end{proof}

\begin{claim}
If $(|A_1(x)|,|N^2(x)|)=(3,3)$, then $G$ is isomorphic to $M_{4,1,1,1,1},M_{4,2,1,1}$ or $M_{4,3,1}$. 
\end{claim}

\begin{proof}
In this case $\ell(G)=n-1$ and $\Li(G)=\Li^x_2\cup \Li^x_3\cup L_{A_2}\cup L_2\cup \{\li{xu}\}$. Moreover, as $|A_1(x)|=|N^2(x)|$, we have that $N^3(x)=A_1(x)$ and $N^2(x)=B_1(x)$. Let $A_1=\{w^*,w,w'\}$, $v=f(w)$ and $v'=f(w')$, thus $N^2(x)=\{v^*,v,v'\}$. To prove the claim, it is sufficient to show that $|E(G[N^3(x)])|\neq 2$. We proceed by contradiction, supposing that $|E(G[N^3(x)])|=2$. We may assume that $E(G[N^3(x)])=\{w^*w,ww'\}$, and we consider the line $l=\li{w^*w'}$. Since $d(w^*,w')=2$, we obtain that $l\cap \left (N^2(x)-v^*\right )=\emptyset$, therefore $l\notin  L_2$. Observing that $x\notin l$, we conclude that $l\notin \Li(G)$, leading to a contradiction.
\end{proof}

This concludes the proof of Proposition \ref{PropR1A_1nonempty}. We will now prove the following Proposition.

\begin{prop}
If $A_1(x)=\emptyset$, then $G$ is isomorphic to $G_d$.
\label{PropR1nonemptyA1empty}
\end{prop}

\subsubsection*{Proof of Proposition \ref{PropR1nonemptyA1empty}}

We now assume that $A_1(x)=\emptyset$. If $|N^2(x)|=1$, $G$ is isomorphic to $M_{2,1,1}$ with $A_1(x)\neq \emptyset$, thus it must hold that $|N^2(x)|\geq 2$. Clearly $w^*\in A_2(x)$, that is, $\li{xv^*}=\li{xw^*}\neq \li{w^*u}$. The set $\Li^x_2\cup \Li^x_3\cup L_{A_2} \cup L^*$ contains $n-3+ |N^2(x)|$ lines. It follows that $|N^2(x)|=2$ and $\ell(G)=n-1$. Let $N^2(x)=\{v^*,v\}$, hence $L^*=\{\li{w^*v}\}$ and $\Li(G)=\Li^x_2\cup \Li^x_3\cup L_{A_2}\cup \{\li{w^*v}\}$. As $v\notin \li{xv^*}=\li{xw^*}$, we have that $V\notin \Li^x_2\cup \Li^x_3$, and since $x\notin \li{w^*v}$, the line  $V=\li{xu}$ must belong to $L_{A_2}$. This implies that $\li{w^*u}=V$, $N^3(x)=\{w^*\}$ and $d(w^*,v)=3$, concluding that $G$ is isomorphic to $G_d$. 

\subsection{\texorpdfstring{$R_1$}{R1} is empty}

In this section we assume that $R_1=\emptyset$, that is, $\forall\, x\in R$, $d(x)\geq 2$. Let $x$ be a vertex in $R$ maximizing $|A(\cdot )|$. 

\begin{prop}
If $A_1(x)\neq \emptyset$, then $G$ is isomorphic to $M_{4,2,2}$.
\label{PropR1emptyA1nonempty}
\end{prop}
\subsubsection*{Proof of Proposition \ref{PropR1emptyA1nonempty}}

We assume that $A_1(x)\neq \emptyset$. Consider $w\in A_1(x)$, $v=f(w)$ and $u\in N(x)$ such that $\li{wu}=\li{xv}=\li{xw}$. Note that $\li{wu}=\li{wx}\in \Li^w_2 \cap \Li^w_3$, thus $x\in A(w)$ and $u\in B(w)$, in particular, $w\in R$ and $d(w)\geq 2$. Since $x\in \li{wu}$, we obtain that $d(w,u)=2$ and $vu\in E$. We observe that $\left (N(v)-u\right )\cap N(x)\subseteq \li{xw}\setminus\li{wu}=\emptyset$, therefore $N(v)\cap N(x)=\{u\}$. In particular, $N^2(w)\cap N(x)=\{u\}$, $N^3(w)\cap N(x)=N(x)-u$ and $|N^3(w)\cap N(x)|=|N(x)|-1$. Since $N(u)\cap N(x) \subseteq \li{wu}\setminus \li{xw}=\emptyset$ and $N^3(v)\cap N(x)\subseteq \li{xv}\setminus \li{xw}=\emptyset$, it follows that $N(u)\cap N(x)=\emptyset$ and $N(x)-u\subseteq N^2(v)$. Vertices $w,v$ and $u$ are illustrated in the following diagram.
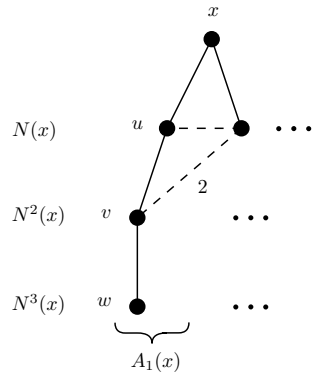
\begin{figure}[H]
\begin{center}      
\scalebox{.75}{\tikzset{every picture/.style={line width=0.75pt}} 

\begin{tikzpicture}[x=0.75pt,y=0.75pt,yscale=-1,xscale=1]

\draw  [fill={rgb, 255:red, 0; green, 0; blue, 0 }  ,fill opacity=1 ] (270,45) .. controls (270,42.24) and (272.24,40) .. (275,40) .. controls (277.76,40) and (280,42.24) .. (280,45) .. controls (280,47.76) and (277.76,50) .. (275,50) .. controls (272.24,50) and (270,47.76) .. (270,45) -- cycle ;
\draw  [fill={rgb, 255:red, 0; green, 0; blue, 0 }  ,fill opacity=1 ] (240,105) .. controls (240,102.24) and (242.24,100) .. (245,100) .. controls (247.76,100) and (250,102.24) .. (250,105) .. controls (250,107.76) and (247.76,110) .. (245,110) .. controls (242.24,110) and (240,107.76) .. (240,105) -- cycle ;
\draw    (275,45) -- (245,105) ;
\draw  [fill={rgb, 255:red, 0; green, 0; blue, 0 }  ,fill opacity=1 ] (220,165.14) .. controls (220,162.38) and (222.24,160.14) .. (225,160.14) .. controls (227.76,160.14) and (230,162.38) .. (230,165.14) .. controls (230,167.9) and (227.76,170.14) .. (225,170.14) .. controls (222.24,170.14) and (220,167.9) .. (220,165.14) -- cycle ;
\draw  [fill={rgb, 255:red, 0; green, 0; blue, 0 }  ,fill opacity=1 ] (220,225) .. controls (220,222.24) and (222.24,220) .. (225,220) .. controls (227.76,220) and (230,222.24) .. (230,225) .. controls (230,227.76) and (227.76,230) .. (225,230) .. controls (222.24,230) and (220,227.76) .. (220,225) -- cycle ;
\draw    (225,165.14) -- (225,225) ;
\draw   (210,235.8) .. controls (210.01,240.47) and (212.35,242.79) .. (217.02,242.77) -- (224.86,242.75) .. controls (231.53,242.73) and (234.87,245.05) .. (234.89,249.72) .. controls (234.87,245.05) and (238.19,242.71) .. (244.86,242.68)(241.86,242.69) -- (252.7,242.66) .. controls (257.37,242.64) and (259.69,240.3) .. (259.68,235.63) ;
\draw    (245,105) -- (225,165.14) ;
\draw  [fill={rgb, 255:red, 0; green, 0; blue, 0 }  ,fill opacity=1 ] (290,105) .. controls (290,102.24) and (292.24,100) .. (295,100) .. controls (297.76,100) and (300,102.24) .. (300,105) .. controls (300,107.76) and (297.76,110) .. (295,110) .. controls (292.24,110) and (290,107.76) .. (290,105) -- cycle ;
\draw  [fill={rgb, 255:red, 0; green, 0; blue, 0 }  ,fill opacity=1 ] (321.29,105.25) .. controls (321.29,104.59) and (320.76,104.05) .. (320.09,104.05) .. controls (319.43,104.05) and (318.89,104.59) .. (318.89,105.25) .. controls (318.89,105.91) and (319.43,106.45) .. (320.09,106.45) .. controls (320.76,106.45) and (321.29,105.91) .. (321.29,105.25) -- cycle ;
\draw  [fill={rgb, 255:red, 0; green, 0; blue, 0 }  ,fill opacity=1 ] (331.42,105.2) .. controls (331.42,104.54) and (330.88,104) .. (330.22,104) .. controls (329.55,104) and (329.02,104.54) .. (329.02,105.2) .. controls (329.02,105.86) and (329.55,106.4) .. (330.22,106.4) .. controls (330.88,106.4) and (331.42,105.86) .. (331.42,105.2) -- cycle ;
\draw  [fill={rgb, 255:red, 0; green, 0; blue, 0 }  ,fill opacity=1 ] (341.43,105.25) .. controls (341.43,104.59) and (340.89,104.05) .. (340.23,104.05) .. controls (339.57,104.05) and (339.03,104.59) .. (339.03,105.25) .. controls (339.03,105.91) and (339.57,106.45) .. (340.23,106.45) .. controls (340.89,106.45) and (341.43,105.91) .. (341.43,105.25) -- cycle ;
\draw  [dash pattern={on 4.5pt off 4.5pt}]  (295,105) -- (245,105) ;
\draw  [dash pattern={on 4.5pt off 4.5pt}]  (295,105) -- (225,165.14) ;
\draw    (275,45) -- (295,105) ;
\draw  [fill={rgb, 255:red, 0; green, 0; blue, 0 }  ,fill opacity=1 ] (292.4,225.25) .. controls (292.4,224.59) and (291.86,224.05) .. (291.2,224.05) .. controls (290.54,224.05) and (290,224.59) .. (290,225.25) .. controls (290,225.91) and (290.54,226.45) .. (291.2,226.45) .. controls (291.86,226.45) and (292.4,225.91) .. (292.4,225.25) -- cycle ;
\draw  [fill={rgb, 255:red, 0; green, 0; blue, 0 }  ,fill opacity=1 ] (302.52,225.2) .. controls (302.52,224.54) and (301.98,224) .. (301.32,224) .. controls (300.66,224) and (300.12,224.54) .. (300.12,225.2) .. controls (300.12,225.86) and (300.66,226.4) .. (301.32,226.4) .. controls (301.98,226.4) and (302.52,225.86) .. (302.52,225.2) -- cycle ;
\draw  [fill={rgb, 255:red, 0; green, 0; blue, 0 }  ,fill opacity=1 ] (312.54,225.25) .. controls (312.54,224.59) and (312,224.05) .. (311.34,224.05) .. controls (310.68,224.05) and (310.14,224.59) .. (310.14,225.25) .. controls (310.14,225.91) and (310.68,226.45) .. (311.34,226.45) .. controls (312,226.45) and (312.54,225.91) .. (312.54,225.25) -- cycle ;
\draw  [fill={rgb, 255:red, 0; green, 0; blue, 0 }  ,fill opacity=1 ] (292.4,165.25) .. controls (292.4,164.59) and (291.86,164.05) .. (291.2,164.05) .. controls (290.54,164.05) and (290,164.59) .. (290,165.25) .. controls (290,165.91) and (290.54,166.45) .. (291.2,166.45) .. controls (291.86,166.45) and (292.4,165.91) .. (292.4,165.25) -- cycle ;
\draw  [fill={rgb, 255:red, 0; green, 0; blue, 0 }  ,fill opacity=1 ] (302.52,165.2) .. controls (302.52,164.54) and (301.98,164) .. (301.32,164) .. controls (300.66,164) and (300.12,164.54) .. (300.12,165.2) .. controls (300.12,165.86) and (300.66,166.4) .. (301.32,166.4) .. controls (301.98,166.4) and (302.52,165.86) .. (302.52,165.2) -- cycle ;
\draw  [fill={rgb, 255:red, 0; green, 0; blue, 0 }  ,fill opacity=1 ] (312.54,165.25) .. controls (312.54,164.59) and (312,164.05) .. (311.34,164.05) .. controls (310.68,164.05) and (310.14,164.59) .. (310.14,165.25) .. controls (310.14,165.91) and (310.68,166.45) .. (311.34,166.45) .. controls (312,166.45) and (312.54,165.91) .. (312.54,165.25) -- cycle ;

\draw (271,22.4) node [anchor=north west][inner sep=0.75pt]    {$x$};
\draw (139.29,98.83) node [anchor=north west][inner sep=0.75pt]    {$N( x)$};
\draw (139.04,155.08) node [anchor=north west][inner sep=0.75pt]    {$N^{2}( x)$};
\draw (139.04,215.08) node [anchor=north west][inner sep=0.75pt]    {$N^{3}( x)$};
\draw (199.1,158.29) node [anchor=north west][inner sep=0.75pt]    {$v$};
\draw (196.52,217.54) node [anchor=north west][inner sep=0.75pt]    {$w$};
\draw (219.08,254.11) node [anchor=north west][inner sep=0.75pt]    {$A_{1}( x)$};
\draw (219.79,98.21) node [anchor=north west][inner sep=0.75pt]    {$u$};
\draw (264.22,138.4) node [anchor=north west][inner sep=0.75pt]    {$2$};

\end{tikzpicture}}
\end{center}
\caption{Diagram of $G$. Dashed lines represent missing edges. Numbers indicate distances between the corresponding vertices.} \label{fig:R1emptyDiagram1}
\end{figure}

We now consider the set of lines $L_{A_1}=\{\li{wu'}: w\in A_1(x), \, u'\in N^3(w)\cap N(x)\}$. As no line in $L_{A_1}$ contains $x$, the set $L_{A_1}$ is disjoint from $\Li^x_2, \Li^x_3$ and $L_{A_2}$.

\begin{claim}
     $|L_{A_1}|=|A_1(x)|\cdot |N(x)|-|A_1(x)|$.
\end{claim}

\begin{proof}
For every $w\in A_1(x)$, it holds that $|N^3(w)\cap N(x)|= |N(x)|-1$, therefore it suffices to prove that for distinct $w,w'\in A_1(x)$, $u' \in N^3(w)\cap N(x)$ and $u'' \in N^3(w')\cap N(x)$, the lines $\li{wu'}$ and $\li{w'u''}$ are distinct. We proceed by contradiction, supposing that $\li{wu'}=\li{w'u''}$. It is immediate that $u'\neq u''$. Let $v=f(w)$ and $u\in N(x)$ such that $\li{wu}=\li{xv}=\li{xw}$. Given that $u''\in \li{wu'}$, we have that $d(w,u'')=2$ and $u''u'\in E$. This implies that $u''\in N^2(w)\cap N(x)=\{u\}$, thus $u''=u$ and $u'\in N(u)\cap N(x)=\emptyset$, which is a contradiction.
\end{proof}

Each vertex $w\in A_1(x)$ has an associated vertex $u\in N(x)$ such that $\li{wu}=\li{xw}$. We consider these vertices as the set $C=\{u\in N(x):\exists\, w\in A_1(x) \text{ such that } \li{wu}=\li{xw}\}$.

\begin{claim}
If $u^*\in C$, then $\li{xu^*}\notin \Li^x_2\cup \Li^x_3 \cup L_{A_1}\cup L_{A_2}$.
\label{R1emptyC}
\end{claim}

\begin{proof}
Consider $u^*\in C$, $w^*\in A_1(x)$ and $v^*=f(w^*)$ such that $\li{w^*u^*}=\li{xv^*}=\li{xw^*}$. Recall that $d(w^*,u^*)=2$,  $v^*u^*\in E$, $N(u^*)\cap N(x)=\emptyset$ and $N(x)-u^*\subseteq N^2(v^*)$. The fact that $N(x)-u^*\subseteq \li{xu^*}\setminus \li{xv^*}$, together with $d(x)\geq 2$, implies that $\li{xu^*}\neq \li{xv^*}=\li{xw^*}$. As $v^*,w^*\in \li{xu^*}$, the line $\li{xu^*}$ does not belong to $\Li^x_2\cup \Li^x_3$. The line $\li{xu^*}$ contains $x$, thus it also does not belong to $L_{A_1}$. It remains to show that $\li{xu^*}\notin L_{A_2}$. By contradiction, suppose that $\li{xu^*}=\li{wu}\in L_{A_2}$. Since $w\in \li{xu^*}$, we obtain that $d(w,u^*)=2$ and consequently $u=u^*$, leading to the contradiction $w^*\in \li{xu^*}\setminus \li{wu}$.
\end{proof}

We now fix $w^*\in A_1(x)$, $v^*=f(w^*)$ and $u^*\in N(x)$ such that $\li{w^*u^*}=\li{xv^*}=\li{xw^*}$. It holds that $x\in A(w^*)$, $u^*\in B(w^*)$, $d(w^*)\geq 2$, $N(v^*)\cap N(x)=\{u^*\}$, $N^2(w^*)\cap N(x)=\{u^*\}$ , $N^3(w^*)\cap N(x)=N(x)-u^* $, $N(u^*)\cap N(x)=\emptyset$, and $N(x)-u^*\subseteq N^2(v^*)$. Since $u^*\in C$, the set $\Li^x_2\cup \Li^x_3\cup  L_{A_1}\cup L_{A_2}\cup \{\li{xu^*}\}$ contains at least 
{\small $$
|N^3(x)|+|N^2(x)|- |A(x)|+ |A_1(x)|\cdot |N(x)|-|A_1(x)|+|A_2(x)|+1=|N^3(x)|+|N^2(x)|+ |A_1(x)|\cdot |N(x)|-2|A_1(x)|+1
$$}

lines. As $\ell(G)\leq n-1=|N(x)|+|N^2(x)|+|N^3(x)|$, we have that $(|A_1(x)| -1)\cdot (|N(x)|-2)\leq 1$, therefore $|A_1(x)|\leq 2$ or $ |N(x)|\leq 2$.

\begin{claim}
$|A_1(x)|\neq  1$.
\end{claim}

\begin{proof}
For the sake of contradiction, suppose that $A_1(x)=\{w^*\}$, observing that $L_{A_1}=\{\li{w^*u}:u\in N(x)-u^*\}$. Note that $ N(w^*)-v^*\subseteq N^3(x)-w^*$, and recalling that $d(w^*)\geq 2$, let $w\in N(w^*)-v^*$. Since $N^3(u^*)\cap N(w^*)\subseteq \li{w^*u^*}\setminus \li{xw^*}=\emptyset$, we have that $d(w,u^*)=2$. Consider $v\in N(w)\cap N(u^*)$, noting that $v\in N^2(x)-v^*$. $G$ is represented in the following diagram.
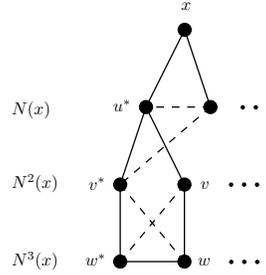
\begin{center}
\scalebox{.65}{\tikzset{every picture/.style={line width=0.75pt}} 

\begin{tikzpicture}[x=0.75pt,y=0.75pt,yscale=-1,xscale=1]

\draw  [fill={rgb, 255:red, 0; green, 0; blue, 0 }  ,fill opacity=1 ] (260.17,65) .. controls (260.17,62.24) and (262.41,60) .. (265.17,60) .. controls (267.93,60) and (270.17,62.24) .. (270.17,65) .. controls (270.17,67.76) and (267.93,70) .. (265.17,70) .. controls (262.41,70) and (260.17,67.76) .. (260.17,65) -- cycle ;
\draw  [fill={rgb, 255:red, 0; green, 0; blue, 0 }  ,fill opacity=1 ] (230.17,125) .. controls (230.17,122.24) and (232.41,120) .. (235.17,120) .. controls (237.93,120) and (240.17,122.24) .. (240.17,125) .. controls (240.17,127.76) and (237.93,130) .. (235.17,130) .. controls (232.41,130) and (230.17,127.76) .. (230.17,125) -- cycle ;
\draw    (265.17,65) -- (235.17,125) ;
\draw  [fill={rgb, 255:red, 0; green, 0; blue, 0 }  ,fill opacity=1 ] (210.17,185.14) .. controls (210.17,182.38) and (212.41,180.14) .. (215.17,180.14) .. controls (217.93,180.14) and (220.17,182.38) .. (220.17,185.14) .. controls (220.17,187.9) and (217.93,190.14) .. (215.17,190.14) .. controls (212.41,190.14) and (210.17,187.9) .. (210.17,185.14) -- cycle ;
\draw  [fill={rgb, 255:red, 0; green, 0; blue, 0 }  ,fill opacity=1 ] (210.17,245) .. controls (210.17,242.24) and (212.41,240) .. (215.17,240) .. controls (217.93,240) and (220.17,242.24) .. (220.17,245) .. controls (220.17,247.76) and (217.93,250) .. (215.17,250) .. controls (212.41,250) and (210.17,247.76) .. (210.17,245) -- cycle ;
\draw    (215.17,185.14) -- (215.17,245) ;
\draw    (235.17,125) -- (215.17,185.14) ;
\draw  [fill={rgb, 255:red, 0; green, 0; blue, 0 }  ,fill opacity=1 ] (280.17,125) .. controls (280.17,122.24) and (282.41,120) .. (285.17,120) .. controls (287.93,120) and (290.17,122.24) .. (290.17,125) .. controls (290.17,127.76) and (287.93,130) .. (285.17,130) .. controls (282.41,130) and (280.17,127.76) .. (280.17,125) -- cycle ;
\draw  [fill={rgb, 255:red, 0; green, 0; blue, 0 }  ,fill opacity=1 ] (311.46,125.25) .. controls (311.46,124.59) and (310.92,124.05) .. (310.26,124.05) .. controls (309.6,124.05) and (309.06,124.59) .. (309.06,125.25) .. controls (309.06,125.91) and (309.6,126.45) .. (310.26,126.45) .. controls (310.92,126.45) and (311.46,125.91) .. (311.46,125.25) -- cycle ;
\draw  [fill={rgb, 255:red, 0; green, 0; blue, 0 }  ,fill opacity=1 ] (321.58,125.2) .. controls (321.58,124.54) and (321.05,124) .. (320.38,124) .. controls (319.72,124) and (319.18,124.54) .. (319.18,125.2) .. controls (319.18,125.86) and (319.72,126.4) .. (320.38,126.4) .. controls (321.05,126.4) and (321.58,125.86) .. (321.58,125.2) -- cycle ;
\draw  [fill={rgb, 255:red, 0; green, 0; blue, 0 }  ,fill opacity=1 ] (331.6,125.25) .. controls (331.6,124.59) and (331.06,124.05) .. (330.4,124.05) .. controls (329.74,124.05) and (329.2,124.59) .. (329.2,125.25) .. controls (329.2,125.91) and (329.74,126.45) .. (330.4,126.45) .. controls (331.06,126.45) and (331.6,125.91) .. (331.6,125.25) -- cycle ;
\draw  [dash pattern={on 4.5pt off 4.5pt}]  (285.17,125) -- (235.17,125) ;
\draw  [dash pattern={on 4.5pt off 4.5pt}]  (285.17,125) -- (215.17,185.14) ;
\draw    (265.17,65) -- (285.17,125) ;
\draw  [fill={rgb, 255:red, 0; green, 0; blue, 0 }  ,fill opacity=1 ] (302.4,185.25) .. controls (302.4,184.59) and (301.86,184.05) .. (301.2,184.05) .. controls (300.54,184.05) and (300,184.59) .. (300,185.25) .. controls (300,185.91) and (300.54,186.45) .. (301.2,186.45) .. controls (301.86,186.45) and (302.4,185.91) .. (302.4,185.25) -- cycle ;
\draw  [fill={rgb, 255:red, 0; green, 0; blue, 0 }  ,fill opacity=1 ] (312.52,185.2) .. controls (312.52,184.54) and (311.98,184) .. (311.32,184) .. controls (310.66,184) and (310.12,184.54) .. (310.12,185.2) .. controls (310.12,185.86) and (310.66,186.4) .. (311.32,186.4) .. controls (311.98,186.4) and (312.52,185.86) .. (312.52,185.2) -- cycle ;
\draw  [fill={rgb, 255:red, 0; green, 0; blue, 0 }  ,fill opacity=1 ] (322.54,185.25) .. controls (322.54,184.59) and (322,184.05) .. (321.34,184.05) .. controls (320.68,184.05) and (320.14,184.59) .. (320.14,185.25) .. controls (320.14,185.91) and (320.68,186.45) .. (321.34,186.45) .. controls (322,186.45) and (322.54,185.91) .. (322.54,185.25) -- cycle ;
\draw  [fill={rgb, 255:red, 0; green, 0; blue, 0 }  ,fill opacity=1 ] (260,185.14) .. controls (260,182.38) and (262.24,180.14) .. (265,180.14) .. controls (267.76,180.14) and (270,182.38) .. (270,185.14) .. controls (270,187.9) and (267.76,190.14) .. (265,190.14) .. controls (262.24,190.14) and (260,187.9) .. (260,185.14) -- cycle ;
\draw  [fill={rgb, 255:red, 0; green, 0; blue, 0 }  ,fill opacity=1 ] (260,245) .. controls (260,242.24) and (262.24,240) .. (265,240) .. controls (267.76,240) and (270,242.24) .. (270,245) .. controls (270,247.76) and (267.76,250) .. (265,250) .. controls (262.24,250) and (260,247.76) .. (260,245) -- cycle ;
\draw    (265,185.14) -- (265,245) ;
\draw    (215.17,245) -- (265,245) ;
\draw  [dash pattern={on 4.5pt off 4.5pt}]  (265,185.14) -- (215.17,245) ;
\draw  [dash pattern={on 4.5pt off 4.5pt}]  (215.17,185.14) -- (265,245) ;
\draw    (235.17,125) -- (265,185.14) ;
\draw  [fill={rgb, 255:red, 0; green, 0; blue, 0 }  ,fill opacity=1 ] (302.57,245.25) .. controls (302.57,244.59) and (302.03,244.05) .. (301.37,244.05) .. controls (300.71,244.05) and (300.17,244.59) .. (300.17,245.25) .. controls (300.17,245.91) and (300.71,246.45) .. (301.37,246.45) .. controls (302.03,246.45) and (302.57,245.91) .. (302.57,245.25) -- cycle ;
\draw  [fill={rgb, 255:red, 0; green, 0; blue, 0 }  ,fill opacity=1 ] (312.69,245.2) .. controls (312.69,244.54) and (312.15,244) .. (311.49,244) .. controls (310.83,244) and (310.29,244.54) .. (310.29,245.2) .. controls (310.29,245.86) and (310.83,246.4) .. (311.49,246.4) .. controls (312.15,246.4) and (312.69,245.86) .. (312.69,245.2) -- cycle ;
\draw  [fill={rgb, 255:red, 0; green, 0; blue, 0 }  ,fill opacity=1 ] (322.71,245.25) .. controls (322.71,244.59) and (322.17,244.05) .. (321.51,244.05) .. controls (320.84,244.05) and (320.31,244.59) .. (320.31,245.25) .. controls (320.31,245.91) and (320.84,246.45) .. (321.51,246.45) .. controls (322.17,246.45) and (322.71,245.91) .. (322.71,245.25) -- cycle ;

\draw (261.17,42.4) node [anchor=north west][inner sep=0.75pt]    {$x$};
\draw (129.45,118.83) node [anchor=north west][inner sep=0.75pt]    {$N( x)$};
\draw (129.45,175.63) node [anchor=north west][inner sep=0.75pt]    {$N^{2}( x)$};
\draw (129.45,235.63) node [anchor=north west][inner sep=0.75pt]    {$N^{3}( x)$};
\draw (189.26,178.29) node [anchor=north west][inner sep=0.75pt]    {$v^{*}$};
\draw (186.69,237.54) node [anchor=north west][inner sep=0.75pt]    {$w^{*}$};
\draw (208.36,117.81) node [anchor=north west][inner sep=0.75pt]    {$u^{*}$};
\draw (275.87,181.65) node [anchor=north west][inner sep=0.75pt]    {$v$};
\draw (274.05,240.4) node [anchor=north west][inner sep=0.75pt]    {$w$};

\end{tikzpicture}}
\end{center}
As $|A_1(x)|=1$, the set $\Li^x_2\cup \Li^x_3\cup  L_{A_1} \cup L_{A_2}\cup \{\li{xu^*}\}$ contains at least $n-2$ lines. Consider the line $l_1=\li{wv^*}$, observing that $x\notin l_1$ and $l_1\cap \left (N(x)-u^*\right )=\emptyset$. Given that $d(w,u^*)=d(w,v^*)=2$, the vertex $u^*$ does not belong to the line $l_1$, hence $l_1\cap \left ( N(x)+x\right) =\emptyset$. It follows that $l_1\notin \Li^x_2\cup \Li^x_3\cup L_{A_1}\cup L_{A_2}\cup \{\li{xu^*}\}$, therefore $\ell(G)=n-1$ and $\Li(G)=\Li^x_2\cup \Li^x_3\cup L_{A_1}\cup L_{A_2}\cup \{\li{xu^*},l_1\}$. We will now show that $A_2(x)$ is empty by contradiction. Suppose $A_2(x)$ is not empty and consider $w'\in A_2(x)$. To count the $n-1$ lines, we previously used the lower bound $|N^2(w')\cap N(x)|\geq 1$, thus $|N^2(w')\cap N(x)|=1$. As $d(x)\geq 2$, there exists $u'\in N^3(w')\cap N(x)$. Consider the line $l_2=\li{w'u'}$. Recall that $l_1\cap N(x)=\emptyset$, therefore $l_2\neq l_1$. Given that the line $l_2$ does not contain $x$, it must hold that $l_2\in L_{A_1}$, thus $w^*\in l_2$. It follows that $d(w^*,u')=2$, therefore $u'\in N^2(w^*)\cap N(x)=\{u^*\}$ and $u'=u^*$. This is a contradiction, since as $N(x)-u^*+w^*\subseteq N^2(u^*)$, no line in $L_{A_1}$ contains $u^*$. It follows that $A_2(x)=\emptyset$, hence, $|A(x)|=|A_1(x)|=1$. As $w^*\in R$ and $x\in A(w^*)$, by the choice of $x$, it holds that $|A(w^*)|=1$. Consequently, $A(w^*)=\{x\}$ and $B(w^*)=\{u^*\}$. Let $u\in N(x)-u^*$. Given that $d(v^*,u)=2$, consider $v'\in N(v^*)\cap N(u)$. Since $N(v^*)\cap N(x)=\{u^*\}$ and $N(u^*)\cap N(x)=\emptyset$, the vertex $v'$ must belong to $N^2(x)$. Now consider the line $l_2=\li{w^*v'}\in \Li^{w^*}_2$, observing that since $B(w^*)=\{u^*\}$, the line $l_2$ does not belong to $L_{A_1}\subseteq \Li^{w^*}_3$. As $u\in N^3(w^*)\cap N(v')$, we have that $u\in l_2$, thus $l_2\neq l_1$. The line $l_2$ does not contain the vertex $x$ implying that $l_2\notin \Li(G)$, yielding a contradiction.
\end{proof}

Consider the set of lines $L^*=\{\li{w^*v}:v\in B_1(x)-v^*\}$. The lines in $L^*$ do not contain  $x$, therefore $L^*$ is disjoint from $\Li^x_2,  \Li^x_3, L_{A_2}$ and $\{\li{xu^*}\}$. 

\begin{claim}
$L^*$ is disjoint from $L_{A_1}$. 
\end{claim}

\begin{proof}
Let $l=\li{w^*\hat{v}}\in L^*$, $\hat{w}=f^{-1}(\hat{v})$ and $\hat{u}\in N(x)$ such that $\li{\hat{w}\hat{u}}=\li{x\hat{v}}=\li{x\hat{w}}$. Recall that $N(\hat{v})\cap N(x)=\{\hat{u}\}$, $N^2(\hat{w})\cap N(x)=\{\hat{u}\}$ , $N^3(\hat{w})\cap N(x)=N(x)-\hat{u} $, $N(\hat{u})\cap N(x)=\emptyset$, and $N(x)-\hat{u}\subseteq N^2(\hat{v})$. We will show that $l\notin L_{A_1}$. By contradiction, assume that $l=\li{wu}\in L_{A_1}$. Since $N^3(w^*)\cap N(u^*)\cap N^2(x)\subseteq \li{w^*u^*}\setminus \li{xv^*}=\emptyset$ and $N(x)-u^*\subseteq N^3(w^*)$, if $d(w^*,\hat{v})=3$, we obtain that $\hat{v}u^*\notin E$ and $l\cap N(x)=\emptyset$. This is a contradiction to the fact that $u\in l$, hence it follows that $d(w^*,\hat{v})=2$.
The line $l=\li{w^*\hat{v}}$ contains the vertex $u$, thus
$u\in N(\hat{v})\cap N^3(w^*)$. Since $N(\hat{v})\cap N(x)=\{\hat{u}\}$ and $d(w^*,u^*)=d(\hat{w},\hat{u})=2$, we obtain that $\hat{u}=u\neq u^*$ and $w\neq \hat{w}$. In particular, we have that $d(w^*,\hat{u})=d(\hat{w},u^*)=3$. $G$ is illustrated in the following diagram.
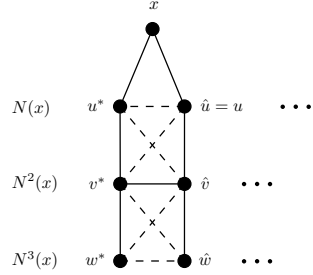
\begin{center}
\scalebox{.65}{\tikzset{every picture/.style={line width=0.75pt}} 

\begin{tikzpicture}[x=0.75pt,y=0.75pt,yscale=-1,xscale=1]

\draw  [fill={rgb, 255:red, 0; green, 0; blue, 0 }  ,fill opacity=1 ] (265.19,54.89) .. controls (265.19,52.13) and (267.43,49.89) .. (270.19,49.89) .. controls (272.95,49.89) and (275.19,52.13) .. (275.19,54.89) .. controls (275.19,57.66) and (272.95,59.89) .. (270.19,59.89) .. controls (267.43,59.89) and (265.19,57.66) .. (265.19,54.89) -- cycle ;
\draw  [fill={rgb, 255:red, 0; green, 0; blue, 0 }  ,fill opacity=1 ] (240,115) .. controls (240,112.24) and (242.24,110) .. (245,110) .. controls (247.76,110) and (250,112.24) .. (250,115) .. controls (250,117.76) and (247.76,120) .. (245,120) .. controls (242.24,120) and (240,117.76) .. (240,115) -- cycle ;
\draw    (270.19,54.89) -- (245,115) ;
\draw  [fill={rgb, 255:red, 0; green, 0; blue, 0 }  ,fill opacity=1 ] (240.14,175.14) .. controls (240.14,172.38) and (242.38,170.14) .. (245.14,170.14) .. controls (247.9,170.14) and (250.14,172.38) .. (250.14,175.14) .. controls (250.14,177.9) and (247.9,180.14) .. (245.14,180.14) .. controls (242.38,180.14) and (240.14,177.9) .. (240.14,175.14) -- cycle ;
\draw  [fill={rgb, 255:red, 0; green, 0; blue, 0 }  ,fill opacity=1 ] (240.14,235) .. controls (240.14,232.24) and (242.38,230) .. (245.14,230) .. controls (247.9,230) and (250.14,232.24) .. (250.14,235) .. controls (250.14,237.76) and (247.9,240) .. (245.14,240) .. controls (242.38,240) and (240.14,237.76) .. (240.14,235) -- cycle ;
\draw    (245.14,175.14) -- (245.14,235) ;
\draw  [fill={rgb, 255:red, 0; green, 0; blue, 0 }  ,fill opacity=1 ] (290,115) .. controls (290,112.24) and (292.24,110) .. (295,110) .. controls (297.76,110) and (300,112.24) .. (300,115) .. controls (300,117.76) and (297.76,120) .. (295,120) .. controls (292.24,120) and (290,117.76) .. (290,115) -- cycle ;
\draw    (270.19,54.89) -- (295,115) ;
\draw  [fill={rgb, 255:red, 0; green, 0; blue, 0 }  ,fill opacity=1 ] (289.97,175.14) .. controls (289.97,172.38) and (292.21,170.14) .. (294.97,170.14) .. controls (297.73,170.14) and (299.97,172.38) .. (299.97,175.14) .. controls (299.97,177.9) and (297.73,180.14) .. (294.97,180.14) .. controls (292.21,180.14) and (289.97,177.9) .. (289.97,175.14) -- cycle ;
\draw  [fill={rgb, 255:red, 0; green, 0; blue, 0 }  ,fill opacity=1 ] (289.97,235) .. controls (289.97,232.24) and (292.21,230) .. (294.97,230) .. controls (297.73,230) and (299.97,232.24) .. (299.97,235) .. controls (299.97,237.76) and (297.73,240) .. (294.97,240) .. controls (292.21,240) and (289.97,237.76) .. (289.97,235) -- cycle ;
\draw    (294.97,175.14) -- (294.97,235) ;
\draw    (245.14,175.14) -- (260.42,175.14) -- (294.97,175.14) ;
\draw  [dash pattern={on 4.5pt off 4.5pt}]  (294.97,175.14) -- (245.14,235) ;
\draw  [dash pattern={on 4.5pt off 4.5pt}]  (245.14,175.14) -- (294.97,235) ;
\draw    (245,115) -- (245.14,175.14) ;
\draw    (295,115) -- (294.97,175.14) ;
\draw  [dash pattern={on 4.5pt off 4.5pt}]  (245.14,175.14) -- (295,115) ;
\draw  [dash pattern={on 4.5pt off 4.5pt}]  (295.06,175.14) -- (245.08,115) ;
\draw  [fill={rgb, 255:red, 0; green, 0; blue, 0 }  ,fill opacity=1 ] (342.4,175.25) .. controls (342.4,174.59) and (341.86,174.05) .. (341.2,174.05) .. controls (340.54,174.05) and (340,174.59) .. (340,175.25) .. controls (340,175.91) and (340.54,176.45) .. (341.2,176.45) .. controls (341.86,176.45) and (342.4,175.91) .. (342.4,175.25) -- cycle ;
\draw  [fill={rgb, 255:red, 0; green, 0; blue, 0 }  ,fill opacity=1 ] (352.52,175.2) .. controls (352.52,174.54) and (351.98,174) .. (351.32,174) .. controls (350.66,174) and (350.12,174.54) .. (350.12,175.2) .. controls (350.12,175.86) and (350.66,176.4) .. (351.32,176.4) .. controls (351.98,176.4) and (352.52,175.86) .. (352.52,175.2) -- cycle ;
\draw  [fill={rgb, 255:red, 0; green, 0; blue, 0 }  ,fill opacity=1 ] (362.54,175.25) .. controls (362.54,174.59) and (362,174.05) .. (361.34,174.05) .. controls (360.68,174.05) and (360.14,174.59) .. (360.14,175.25) .. controls (360.14,175.91) and (360.68,176.45) .. (361.34,176.45) .. controls (362,176.45) and (362.54,175.91) .. (362.54,175.25) -- cycle ;
\draw  [fill={rgb, 255:red, 0; green, 0; blue, 0 }  ,fill opacity=1 ] (342.4,235.25) .. controls (342.4,234.59) and (341.86,234.05) .. (341.2,234.05) .. controls (340.54,234.05) and (340,234.59) .. (340,235.25) .. controls (340,235.91) and (340.54,236.45) .. (341.2,236.45) .. controls (341.86,236.45) and (342.4,235.91) .. (342.4,235.25) -- cycle ;
\draw  [fill={rgb, 255:red, 0; green, 0; blue, 0 }  ,fill opacity=1 ] (352.52,235.2) .. controls (352.52,234.54) and (351.98,234) .. (351.32,234) .. controls (350.66,234) and (350.12,234.54) .. (350.12,235.2) .. controls (350.12,235.86) and (350.66,236.4) .. (351.32,236.4) .. controls (351.98,236.4) and (352.52,235.86) .. (352.52,235.2) -- cycle ;
\draw  [fill={rgb, 255:red, 0; green, 0; blue, 0 }  ,fill opacity=1 ] (362.54,235.25) .. controls (362.54,234.59) and (362,234.05) .. (361.34,234.05) .. controls (360.68,234.05) and (360.14,234.59) .. (360.14,235.25) .. controls (360.14,235.91) and (360.68,236.45) .. (361.34,236.45) .. controls (362,236.45) and (362.54,235.91) .. (362.54,235.25) -- cycle ;
\draw  [fill={rgb, 255:red, 0; green, 0; blue, 0 }  ,fill opacity=1 ] (372.4,115.25) .. controls (372.4,114.59) and (371.86,114.05) .. (371.2,114.05) .. controls (370.54,114.05) and (370,114.59) .. (370,115.25) .. controls (370,115.91) and (370.54,116.45) .. (371.2,116.45) .. controls (371.86,116.45) and (372.4,115.91) .. (372.4,115.25) -- cycle ;
\draw  [fill={rgb, 255:red, 0; green, 0; blue, 0 }  ,fill opacity=1 ] (382.52,115.2) .. controls (382.52,114.54) and (381.98,114) .. (381.32,114) .. controls (380.66,114) and (380.12,114.54) .. (380.12,115.2) .. controls (380.12,115.86) and (380.66,116.4) .. (381.32,116.4) .. controls (381.98,116.4) and (382.52,115.86) .. (382.52,115.2) -- cycle ;
\draw  [fill={rgb, 255:red, 0; green, 0; blue, 0 }  ,fill opacity=1 ] (392.54,115.25) .. controls (392.54,114.59) and (392,114.05) .. (391.34,114.05) .. controls (390.68,114.05) and (390.14,114.59) .. (390.14,115.25) .. controls (390.14,115.91) and (390.68,116.45) .. (391.34,116.45) .. controls (392,116.45) and (392.54,115.91) .. (392.54,115.25) -- cycle ;
\draw  [dash pattern={on 4.5pt off 4.5pt}]  (295,115) -- (245.08,115) ;
\draw  [dash pattern={on 4.5pt off 4.5pt}]  (245.14,235) -- (294.97,235) ;

\draw (266.19,32.29) node [anchor=north west][inner sep=0.75pt]    {$x$};
\draw (159.43,108.83) node [anchor=north west][inner sep=0.75pt]    {$N( x)$};
\draw (159.43,165.63) node [anchor=north west][inner sep=0.75pt]    {$N^{2}( x)$};
\draw (159.43,225.63) node [anchor=north west][inner sep=0.75pt]    {$N^{3}( x)$};
\draw (219.23,168.29) node [anchor=north west][inner sep=0.75pt]    {$v^{*}$};
\draw (216.66,227.54) node [anchor=north west][inner sep=0.75pt]    {$w^{*}$};
\draw (218.19,107.81) node [anchor=north west][inner sep=0.75pt]    {$u^{*}$};
\draw (306.07,168.21) node [anchor=north west][inner sep=0.75pt]    {$\hat{v}$};
\draw (304.62,227.44) node [anchor=north west][inner sep=0.75pt]    {$\hat{w}$};
\draw (305.91,107.69) node [anchor=north west][inner sep=0.75pt]    {$\hat{u} =u$};

\end{tikzpicture}}
\end{center}
As $N^3(u^*)\cap N(w^*)\subseteq \li{w^*u^*}\setminus \li{xw^*}=\emptyset$, it holds that $w^*\hat{w}\notin E$ and then necessarily $v^*\hat{v}\in E$. It follows that $v^*\in \li{w^*\hat{v}}= l=\li{w\hat{u}}$, hence $d(w,v^*)=1$ and $w=w^*$. Now we have that $\li{w^*\hat{v}}=\li{w^*\hat{u}}$, thus $d(w^*,\hat{w})=2$ since otherwise $d(w^*,\hat{w})=3$ and $\hat{w}\in \li{w^*\hat{v}}\setminus\li{w^*\hat{u}}=\emptyset$. Consider $w'\in N(w^*)\cap N(\hat{w})$. Since $d(w',\hat{v})=2$, we obtain that $w'\notin \li{w^*\hat{v}}=\li{w^*\hat{u}}$, therefore $d(w',\hat{u})=3$ and $w'\in \li{\hat{w}\hat{u}}=\li{x\hat{w}}$, leading to a contradiction.
\end{proof}

The set $\Li^x_2\cup \Li^x_3\cup L_{A_1}\cup L_{A_2}\cup L^*\cup \{\li{xu^*}\}$ contains at least 
$$
|N^3(x)|+|N^2(x)|+ |A_1(x)|\cdot |N(x)|-2|A_1(x)|+|L^*|+1
$$

lines. We have that $|B_1(x)|=|A_1(x)|\geq 2$, thus $|L^*|\geq 1$. We observe that for $v\in B_1(x)\cap N(u^*)-v^*$, since $v\notin \li{xv^*}=\li{w^*u^*}$, we obtain that $d(w^*,v)=2$ and therefore $|L^*|\geq |B_1(x)\cap N(u^*)|-1$. 

\begin{claim}
$|A_1(x)|=2$.
\end{claim}

\begin{proof}
By contradiction, assume that $|A_1(x)|\ge 3$. As $|A_1(x)|\leq 2$ or $ |N(x)|\leq 2$, it must hold that $|N(x)|=2$. Let $N(x)=\{u,u^*\}$. We have that $|L^*|=1$, $|B_1(x)\cap N(u^*)|\leq 2$, $\ell(G)=n-1$ and $\Li(G)=\Li^x_2\cup \Li^x_3\cup L_{A_1}\cup L_{A_2}\cup L^*\cup \{\li{xu^*}\}$. Since $|A_1(x)|\geq 3$, we may assume that $|B_1(x)\cap N(u^*)|=2$. Let $B_1(x)\cap N(u^*)=\{v^*,v\}$, hence $L^*=\{\li{w^*v}\}$. Let $w'\in A_1(x)\setminus \{w^*,f^{-1}(v)\}$ and $v'=f(w')$, noting that $v'u^*\notin E$, $v'u\in E$ and $d(w',u^*)=3$. $G$ is depicted in the following diagram.
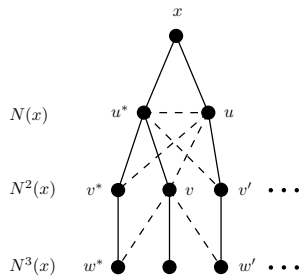
\begin{center}
\scalebox{.65}{\tikzset{every picture/.style={line width=0.75pt}} 

\begin{tikzpicture}[x=0.75pt,y=0.75pt,yscale=-1,xscale=1]

\draw  [fill={rgb, 255:red, 0; green, 0; blue, 0 }  ,fill opacity=1 ] (285.12,65.51) .. controls (285.12,62.75) and (287.36,60.51) .. (290.12,60.51) .. controls (292.89,60.51) and (295.12,62.75) .. (295.12,65.51) .. controls (295.12,68.27) and (292.89,70.51) .. (290.12,70.51) .. controls (287.36,70.51) and (285.12,68.27) .. (285.12,65.51) -- cycle ;
\draw  [fill={rgb, 255:red, 0; green, 0; blue, 0 }  ,fill opacity=1 ] (260.12,125.11) .. controls (260.12,122.35) and (262.36,120.11) .. (265.12,120.11) .. controls (267.89,120.11) and (270.12,122.35) .. (270.12,125.11) .. controls (270.12,127.87) and (267.89,130.11) .. (265.12,130.11) .. controls (262.36,130.11) and (260.12,127.87) .. (260.12,125.11) -- cycle ;
\draw    (290.12,65.51) -- (265.12,125.11) ;
\draw  [fill={rgb, 255:red, 0; green, 0; blue, 0 }  ,fill opacity=1 ] (240.12,185.25) .. controls (240.12,182.49) and (242.36,180.25) .. (245.12,180.25) .. controls (247.89,180.25) and (250.12,182.49) .. (250.12,185.25) .. controls (250.12,188.02) and (247.89,190.25) .. (245.12,190.25) .. controls (242.36,190.25) and (240.12,188.02) .. (240.12,185.25) -- cycle ;
\draw  [fill={rgb, 255:red, 0; green, 0; blue, 0 }  ,fill opacity=1 ] (240.12,245.11) .. controls (240.12,242.35) and (242.36,240.11) .. (245.12,240.11) .. controls (247.89,240.11) and (250.12,242.35) .. (250.12,245.11) .. controls (250.12,247.87) and (247.89,250.11) .. (245.12,250.11) .. controls (242.36,250.11) and (240.12,247.87) .. (240.12,245.11) -- cycle ;
\draw    (245.12,185.25) -- (245.12,245.11) ;
\draw    (265.12,125.11) -- (245.12,185.25) ;
\draw  [fill={rgb, 255:red, 0; green, 0; blue, 0 }  ,fill opacity=1 ] (310.12,125.11) .. controls (310.12,122.35) and (312.36,120.11) .. (315.12,120.11) .. controls (317.89,120.11) and (320.12,122.35) .. (320.12,125.11) .. controls (320.12,127.87) and (317.89,130.11) .. (315.12,130.11) .. controls (312.36,130.11) and (310.12,127.87) .. (310.12,125.11) -- cycle ;
\draw    (290.12,65.51) -- (315.12,125.11) ;
\draw  [fill={rgb, 255:red, 0; green, 0; blue, 0 }  ,fill opacity=1 ] (280,185) .. controls (280,182.24) and (282.24,180) .. (285,180) .. controls (287.76,180) and (290,182.24) .. (290,185) .. controls (290,187.76) and (287.76,190) .. (285,190) .. controls (282.24,190) and (280,187.76) .. (280,185) -- cycle ;
\draw  [fill={rgb, 255:red, 0; green, 0; blue, 0 }  ,fill opacity=1 ] (280,245) .. controls (280,242.24) and (282.24,240) .. (285,240) .. controls (287.76,240) and (290,242.24) .. (290,245) .. controls (290,247.76) and (287.76,250) .. (285,250) .. controls (282.24,250) and (280,247.76) .. (280,245) -- cycle ;
\draw    (285,185) -- (285,244.86) ;
\draw  [dash pattern={on 4.5pt off 4.5pt}]  (285,185) -- (245.12,245.11) ;
\draw  [fill={rgb, 255:red, 0; green, 0; blue, 0 }  ,fill opacity=1 ] (320,245) .. controls (320,242.24) and (322.24,240) .. (325,240) .. controls (327.76,240) and (330,242.24) .. (330,245) .. controls (330,247.76) and (327.76,250) .. (325,250) .. controls (322.24,250) and (320,247.76) .. (320,245) -- cycle ;
\draw  [fill={rgb, 255:red, 0; green, 0; blue, 0 }  ,fill opacity=1 ] (320,185) .. controls (320,182.24) and (322.24,180) .. (325,180) .. controls (327.76,180) and (330,182.24) .. (330,185) .. controls (330,187.76) and (327.76,190) .. (325,190) .. controls (322.24,190) and (320,187.76) .. (320,185) -- cycle ;
\draw    (325,185) -- (325,245) ;
\draw    (265.12,125.11) -- (285,185) ;
\draw    (315.12,125.11) -- (325,185) ;
\draw  [fill={rgb, 255:red, 0; green, 0; blue, 0 }  ,fill opacity=1 ] (364.81,185.36) .. controls (364.81,184.7) and (364.27,184.16) .. (363.61,184.16) .. controls (362.95,184.16) and (362.41,184.7) .. (362.41,185.36) .. controls (362.41,186.02) and (362.95,186.56) .. (363.61,186.56) .. controls (364.27,186.56) and (364.81,186.02) .. (364.81,185.36) -- cycle ;
\draw  [fill={rgb, 255:red, 0; green, 0; blue, 0 }  ,fill opacity=1 ] (374.93,185.31) .. controls (374.93,184.65) and (374.4,184.11) .. (373.73,184.11) .. controls (373.07,184.11) and (372.53,184.65) .. (372.53,185.31) .. controls (372.53,185.97) and (373.07,186.51) .. (373.73,186.51) .. controls (374.4,186.51) and (374.93,185.97) .. (374.93,185.31) -- cycle ;
\draw  [fill={rgb, 255:red, 0; green, 0; blue, 0 }  ,fill opacity=1 ] (384.95,185.36) .. controls (384.95,184.7) and (384.41,184.16) .. (383.75,184.16) .. controls (383.09,184.16) and (382.55,184.7) .. (382.55,185.36) .. controls (382.55,186.02) and (383.09,186.56) .. (383.75,186.56) .. controls (384.41,186.56) and (384.95,186.02) .. (384.95,185.36) -- cycle ;
\draw  [fill={rgb, 255:red, 0; green, 0; blue, 0 }  ,fill opacity=1 ] (364.86,245.36) .. controls (364.86,244.7) and (364.32,244.16) .. (363.66,244.16) .. controls (362.99,244.16) and (362.46,244.7) .. (362.46,245.36) .. controls (362.46,246.02) and (362.99,246.56) .. (363.66,246.56) .. controls (364.32,246.56) and (364.86,246.02) .. (364.86,245.36) -- cycle ;
\draw  [fill={rgb, 255:red, 0; green, 0; blue, 0 }  ,fill opacity=1 ] (374.98,245.31) .. controls (374.98,244.65) and (374.44,244.11) .. (373.78,244.11) .. controls (373.12,244.11) and (372.58,244.65) .. (372.58,245.31) .. controls (372.58,245.97) and (373.12,246.51) .. (373.78,246.51) .. controls (374.44,246.51) and (374.98,245.97) .. (374.98,245.31) -- cycle ;
\draw  [fill={rgb, 255:red, 0; green, 0; blue, 0 }  ,fill opacity=1 ] (384.99,245.36) .. controls (384.99,244.7) and (384.46,244.16) .. (383.79,244.16) .. controls (383.13,244.16) and (382.59,244.7) .. (382.59,245.36) .. controls (382.59,246.02) and (383.13,246.56) .. (383.79,246.56) .. controls (384.46,246.56) and (384.99,246.02) .. (384.99,245.36) -- cycle ;
\draw  [dash pattern={on 4.5pt off 4.5pt}]  (265.12,125.11) -- (325,185) ;
\draw  [dash pattern={on 4.5pt off 4.5pt}]  (285,185) -- (325,245) ;
\draw  [dash pattern={on 4.5pt off 4.5pt}]  (265.12,125.11) -- (315.12,125.11) ;
\draw  [dash pattern={on 4.5pt off 4.5pt}]  (315.12,125.11) -- (245.12,185.25) ;
\draw  [dash pattern={on 4.5pt off 4.5pt}]  (315.12,125.11) -- (285,185) ;

\draw (286.12,42.91) node [anchor=north west][inner sep=0.75pt]    {$x$};
\draw (159.41,118.94) node [anchor=north west][inner sep=0.75pt]    {$N( x)$};
\draw (159.41,175.74) node [anchor=north west][inner sep=0.75pt]    {$N^{2}( x)$};
\draw (159.41,235.74) node [anchor=north west][inner sep=0.75pt]    {$N^{3}( x)$};
\draw (219.22,178.4) node [anchor=north west][inner sep=0.75pt]    {$v^{*}$};
\draw (216.65,237.65) node [anchor=north west][inner sep=0.75pt]    {$w^{*}$};
\draw (238.32,117.92) node [anchor=north west][inner sep=0.75pt]    {$u^{*}$};
\draw (325.82,121.15) node [anchor=north west][inner sep=0.75pt]    {$u$};
\draw (337.06,177.51) node [anchor=north west][inner sep=0.75pt]    {$v'$};
\draw (337.2,236.8) node [anchor=north west][inner sep=0.75pt]    {$w'$};
\draw (295.73,182.15) node [anchor=north west][inner sep=0.75pt]    {$v$};

\end{tikzpicture}}
\end{center}
Now consider the line $l=\li{w^*w'}$. As $x\notin l$, it holds that $l\in L_{A_1}\cup \{\li{w^*v}\}$, and since $l\cap \left( N(x)\cup \{v\}\right )\neq \emptyset$, we obtain that $w^*w'\in E$. It follows that $w'\in \li{w^*u^*}\setminus \li{xw^*}=\emptyset$, which is a contradiction.
\end{proof}

As $|A_1(x)|=2$, the set $\Li^x_2\cup \Li^x_3\cup  L_{A_1}\cup L_{A_2}\cup  L^*\cup \{\li{xu^*}\}$ contains at least 
$$
|N^3(x)|+|N^2(x)|+2|N(x)|-3+|L^*|=n-4+|N(x)|+ |L^*|
$$
lines. Given that $\ell(G)\leq n-1$, it holds that $|N(x)|=2$, $|L^*|=1$, $\ell(G)=n-1$ and $\Li(G)=\Li^x_2\cup \Li^x_3\cup L_{A_1}\cup L_{A_2}\cup L^*\cup \{\li{xu^*}\}$. Let $N(x)=\{u^*,u\}$, $A_1(x)=\{w^*,\hat{w}\}$ and $B_1(x)=\{v^*,\hat{v}\}$. Since $w^*\in \li{xu^*}\setminus \li{xu}$, the line $\li{xu}$ must belong to $\Li^x_2\cup \Li^x_3\cup L_{A_1}\cup L_{A_2}$, therefore $u\notin C$ by Claim \ref{R1emptyC}, concluding that $\li{\hat{w}u^*}=\li{x\hat{v}}=\li{x\hat{w}}$ and $\hat{v}u^*\in E$. In particular $L^*=\{\li{w^*\hat{v}}\}$, $d(w^*,\hat{v})=2$, $L_{A_1}=\{\li{w^*u},\li{\hat{w}u}\}$ and $\forall \, l\in L_{A_1}\cup L^*$, $u^*\notin l$. We also recall that $uu^*,v^*u,\hat{v}u\notin E$.

\begin{claim}
$N^2(x)\setminus \{v^*,\hat{v}\}\subseteq \li{uu^*}$. 
\end{claim}

\begin{proof}
Let $v\in N^2(x)\setminus \{v^*,\hat{v}\}$. The set $N(v)\cap N(x)$ must be nonempty, hence to prove that $v\in \li{uu^*}$, it suffices to show that $N^2(v)\cap N(x)=\emptyset$. By contradiction, suppose there exists $u'\in N^2(v)\cap N(x)$ and consider the line $l=\li{vu'}$. The vertices $x,w^*$ and $\hat{w}$ do not belong to $l$, therefore $l\notin \Li(G)$, yielding a contradiction.
\end{proof}

Given that $d(v^*,u)=2$, consider $v\in N(v^*)\cap N(u)$, observing that $v\in N^2(x)\setminus \{v^*,\hat{v}\}\subseteq \li{uu^*}$ and $d(v,u^*)\leq 2$. $G$ is illustrated in the following diagram.
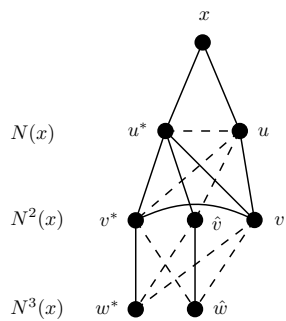
\begin{figure}[H]
\begin{center}      
\scalebox{.75}{\tikzset{every picture/.style={line width=0.75pt}} 

\begin{tikzpicture}[x=0.75pt,y=0.75pt,yscale=-1,xscale=1]

\draw  [fill={rgb, 255:red, 0; green, 0; blue, 0 }  ,fill opacity=1 ] (280.13,55) .. controls (280.13,52.24) and (282.37,50) .. (285.13,50) .. controls (287.89,50) and (290.13,52.24) .. (290.13,55) .. controls (290.13,57.76) and (287.89,60) .. (285.13,60) .. controls (282.37,60) and (280.13,57.76) .. (280.13,55) -- cycle ;
\draw  [fill={rgb, 255:red, 0; green, 0; blue, 0 }  ,fill opacity=1 ] (255.13,114.6) .. controls (255.13,111.84) and (257.37,109.6) .. (260.13,109.6) .. controls (262.89,109.6) and (265.13,111.84) .. (265.13,114.6) .. controls (265.13,117.36) and (262.89,119.6) .. (260.13,119.6) .. controls (257.37,119.6) and (255.13,117.36) .. (255.13,114.6) -- cycle ;
\draw    (285.13,55) -- (260.13,114.6) ;
\draw  [fill={rgb, 255:red, 0; green, 0; blue, 0 }  ,fill opacity=1 ] (235.13,174.74) .. controls (235.13,171.98) and (237.37,169.74) .. (240.13,169.74) .. controls (242.89,169.74) and (245.13,171.98) .. (245.13,174.74) .. controls (245.13,177.5) and (242.89,179.74) .. (240.13,179.74) .. controls (237.37,179.74) and (235.13,177.5) .. (235.13,174.74) -- cycle ;
\draw  [fill={rgb, 255:red, 0; green, 0; blue, 0 }  ,fill opacity=1 ] (235.13,234.6) .. controls (235.13,231.84) and (237.37,229.6) .. (240.13,229.6) .. controls (242.89,229.6) and (245.13,231.84) .. (245.13,234.6) .. controls (245.13,237.36) and (242.89,239.6) .. (240.13,239.6) .. controls (237.37,239.6) and (235.13,237.36) .. (235.13,234.6) -- cycle ;
\draw    (240.13,174.74) -- (240.13,234.6) ;
\draw    (260.13,114.6) -- (240.13,174.74) ;
\draw  [fill={rgb, 255:red, 0; green, 0; blue, 0 }  ,fill opacity=1 ] (305.13,114.6) .. controls (305.13,111.84) and (307.37,109.6) .. (310.13,109.6) .. controls (312.89,109.6) and (315.13,111.84) .. (315.13,114.6) .. controls (315.13,117.36) and (312.89,119.6) .. (310.13,119.6) .. controls (307.37,119.6) and (305.13,117.36) .. (305.13,114.6) -- cycle ;
\draw    (285.13,55) -- (310.13,114.6) ;
\draw  [fill={rgb, 255:red, 0; green, 0; blue, 0 }  ,fill opacity=1 ] (275.01,174.49) .. controls (275.01,171.73) and (277.24,169.49) .. (280.01,169.49) .. controls (282.77,169.49) and (285.01,171.73) .. (285.01,174.49) .. controls (285.01,177.25) and (282.77,179.49) .. (280.01,179.49) .. controls (277.24,179.49) and (275.01,177.25) .. (275.01,174.49) -- cycle ;
\draw  [fill={rgb, 255:red, 0; green, 0; blue, 0 }  ,fill opacity=1 ] (275.01,234.49) .. controls (275.01,231.73) and (277.24,229.49) .. (280.01,229.49) .. controls (282.77,229.49) and (285.01,231.73) .. (285.01,234.49) .. controls (285.01,237.25) and (282.77,239.49) .. (280.01,239.49) .. controls (277.24,239.49) and (275.01,237.25) .. (275.01,234.49) -- cycle ;
\draw    (280.01,174.49) -- (280.01,234.35) ;
\draw  [dash pattern={on 4.5pt off 4.5pt}]  (280.01,174.49) -- (240.13,234.6) ;
\draw  [fill={rgb, 255:red, 0; green, 0; blue, 0 }  ,fill opacity=1 ] (315.01,174.49) .. controls (315.01,171.73) and (317.24,169.49) .. (320.01,169.49) .. controls (322.77,169.49) and (325.01,171.73) .. (325.01,174.49) .. controls (325.01,177.25) and (322.77,179.49) .. (320.01,179.49) .. controls (317.24,179.49) and (315.01,177.25) .. (315.01,174.49) -- cycle ;
\draw    (260.13,114.6) -- (280.01,174.49) ;
\draw    (310.13,114.6) -- (320.01,174.49) ;
\draw    (260.13,114.6) -- (320.01,174.49) ;
\draw  [dash pattern={on 4.5pt off 4.5pt}]  (260.13,114.6) -- (310.13,114.6) ;
\draw  [dash pattern={on 4.5pt off 4.5pt}]  (310.13,114.6) -- (240.13,174.74) ;
\draw  [dash pattern={on 4.5pt off 4.5pt}]  (310.13,114.6) -- (280.01,174.49) ;
\draw    (240.13,174.74) .. controls (269,160.6) and (291,160.2) .. (320.01,174.49) ;
\draw  [dash pattern={on 4.5pt off 4.5pt}]  (240.13,174.74) -- (280.01,234.49) ;
\draw  [dash pattern={on 4.5pt off 4.5pt}]  (240.13,234.6) -- (320.01,174.49) ;
\draw  [dash pattern={on 4.5pt off 4.5pt}]  (320.01,174.49) -- (280.01,234.49) ;

\draw (281.13,32.4) node [anchor=north west][inner sep=0.75pt]    {$x$};
\draw (154.42,108.43) node [anchor=north west][inner sep=0.75pt]    {$N( x)$};
\draw (154.42,165.23) node [anchor=north west][inner sep=0.75pt]    {$N^{2}( x)$};
\draw (154.42,225.23) node [anchor=north west][inner sep=0.75pt]    {$N^{3}( x)$};
\draw (214.22,167.89) node [anchor=north west][inner sep=0.75pt]    {$v^{*}$};
\draw (211.65,227.14) node [anchor=north west][inner sep=0.75pt]    {$w^{*}$};
\draw (233.32,107.41) node [anchor=north west][inner sep=0.75pt]    {$u^{*}$};
\draw (320.82,110.63) node [anchor=north west][inner sep=0.75pt]    {$u$};
\draw (332.06,170.75) node [anchor=north west][inner sep=0.75pt]    {$v$};
\draw (289.9,170.3) node [anchor=north west][inner sep=0.75pt]    {$\hat{v}$};
\draw (290.78,226.96) node [anchor=north west][inner sep=0.75pt]    {$\hat{w}$};

\end{tikzpicture}}
\end{center}
\caption{Diagram of $G$. Dashed lines represent missing edges.} 
\label{fig:R1emptyDiagram2}
\end{figure}

It follows that $vu^*\in E$ and since $\li{uu^*}\cap N^3(x)=\emptyset$ and $v\in \li{uu^*}$, the line $\li{uu^*}$ must be equal to $\li{xv}\in \Li^x_2$. As $N^2(x)\setminus \{v^*,\hat{v}\}\subseteq \li{uu^*}=\li{xv}$, we conclude that $N^2(x)=\{v,v^*,\hat{v}\}$. Given that $\li{xv}\cap N^3(x)=\li{uu^*}\cap N^3(x)=\emptyset$, we also conclude that $N^3(x)=\{w^*,\hat{w}\}$. Moreover, as $d(w^*)\geq 2$, we obtain that $w^*\hat{w}\in E$. Now, we have that $N^2(x)\subseteq N(u^*)$ and $\forall \, l\in \Li(G)$, $x\in l\,\,\lor\,\,u^*\notin l $. It follows that $G[N^2(x)]$ is a complete graph and therefore $G$ is isomorphic to $M_{4,2,2}$. 

This concludes the proof of Proposition \ref{PropR1emptyA1nonempty}. We now proceed to prove the following Proposition.

\begin{prop}
If $A_1(x)= \emptyset$, then $G$ is isomorphic to $G_a$, $G_b$ or $G_c$.
\label{PropR1A1empty}
\end{prop}

\subsubsection*{Proof of Proposition \ref{PropR1A1empty}}

We assume that $A_1(x)=\emptyset$. Let $w^*\in A_2(x)$, $v^*=f(w^*)$ and $u^*\in N^2(w^*)\cap N(x)$. By definition, $\li{xv^*}=\li{xw^*}\neq \li{w^*u^*}$. We consider the sets of lines $\Li^x_2,\Li^x_3,L_{A_2}$ and $L_1^*=\{\li{w^*u}:u\in N^3(w^*)\cap N(x)\}$. Since no line in $L_1^*$ contains $x$, the set $L_1^*$ is disjoint from $\Li^x_2,\Li^x_3$ and $L_{A_2}$. Therefore, the set $\Li^x_2\cup \Li^x_3\cup L_{A_2}\cup L_1^*$ contains 
$$
|N^3(x)|+|N^2(x)|-|A_2(x)|+\sum_{w\in A_2(x)-w^*}|N^2(w)\cap N(x)|+|N^2(w^*)\cap N(x)|+|N^3(w^*)\cap N(x)|\geq n-2
$$ 
lines, where the inequality comes from the fact that $|N^2(w^*)\cap N(x)|+|N^3(w^*)\cap N(x)|=|N(x)|$ and $\sum_{w\in A_2(x)-w^*}|N^2(w)\cap N(x)|\geq |A_2(x)|-1$.

\begin{claim}
$|N^3(x)|=1$.
\end{claim}

\begin{proof}
For the sake of contradiction, suppose that $|N^3(x)|\geq 2$. Let $w$ be the vertex in $ N^3(x)-w^*$ minimizing $d(w^*,\cdot )$.
$G$ is represented in the following diagram.
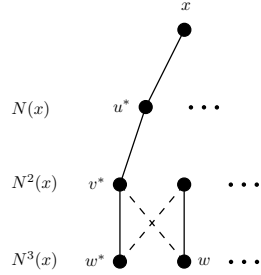
\begin{center}
\scalebox{.65}{\tikzset{every picture/.style={line width=0.75pt}} 

\begin{tikzpicture}[x=0.75pt,y=0.75pt,yscale=-1,xscale=1]

\draw  [fill={rgb, 255:red, 0; green, 0; blue, 0 }  ,fill opacity=1 ] (271.71,65) .. controls (271.71,62.24) and (273.95,60) .. (276.71,60) .. controls (279.48,60) and (281.71,62.24) .. (281.71,65) .. controls (281.71,67.76) and (279.48,70) .. (276.71,70) .. controls (273.95,70) and (271.71,67.76) .. (271.71,65) -- cycle ;
\draw  [fill={rgb, 255:red, 0; green, 0; blue, 0 }  ,fill opacity=1 ] (241.71,125) .. controls (241.71,122.24) and (243.95,120) .. (246.71,120) .. controls (249.48,120) and (251.71,122.24) .. (251.71,125) .. controls (251.71,127.76) and (249.48,130) .. (246.71,130) .. controls (243.95,130) and (241.71,127.76) .. (241.71,125) -- cycle ;
\draw    (276.71,65) -- (246.71,125) ;
\draw  [fill={rgb, 255:red, 0; green, 0; blue, 0 }  ,fill opacity=1 ] (221.71,185.14) .. controls (221.71,182.38) and (223.95,180.14) .. (226.71,180.14) .. controls (229.48,180.14) and (231.71,182.38) .. (231.71,185.14) .. controls (231.71,187.9) and (229.48,190.14) .. (226.71,190.14) .. controls (223.95,190.14) and (221.71,187.9) .. (221.71,185.14) -- cycle ;
\draw  [fill={rgb, 255:red, 0; green, 0; blue, 0 }  ,fill opacity=1 ] (221.71,245) .. controls (221.71,242.24) and (223.95,240) .. (226.71,240) .. controls (229.48,240) and (231.71,242.24) .. (231.71,245) .. controls (231.71,247.76) and (229.48,250) .. (226.71,250) .. controls (223.95,250) and (221.71,247.76) .. (221.71,245) -- cycle ;
\draw    (226.71,185.14) -- (226.71,245) ;
\draw    (246.71,125) -- (226.71,185.14) ;
\draw  [fill={rgb, 255:red, 0; green, 0; blue, 0 }  ,fill opacity=1 ] (283.01,125.25) .. controls (283.01,124.59) and (282.47,124.05) .. (281.81,124.05) .. controls (281.14,124.05) and (280.61,124.59) .. (280.61,125.25) .. controls (280.61,125.91) and (281.14,126.45) .. (281.81,126.45) .. controls (282.47,126.45) and (283.01,125.91) .. (283.01,125.25) -- cycle ;
\draw  [fill={rgb, 255:red, 0; green, 0; blue, 0 }  ,fill opacity=1 ] (293.13,125.2) .. controls (293.13,124.54) and (292.59,124) .. (291.93,124) .. controls (291.27,124) and (290.73,124.54) .. (290.73,125.2) .. controls (290.73,125.86) and (291.27,126.4) .. (291.93,126.4) .. controls (292.59,126.4) and (293.13,125.86) .. (293.13,125.2) -- cycle ;
\draw  [fill={rgb, 255:red, 0; green, 0; blue, 0 }  ,fill opacity=1 ] (303.15,125.25) .. controls (303.15,124.59) and (302.61,124.05) .. (301.95,124.05) .. controls (301.28,124.05) and (300.75,124.59) .. (300.75,125.25) .. controls (300.75,125.91) and (301.28,126.45) .. (301.95,126.45) .. controls (302.61,126.45) and (303.15,125.91) .. (303.15,125.25) -- cycle ;
\draw  [fill={rgb, 255:red, 0; green, 0; blue, 0 }  ,fill opacity=1 ] (313.95,185.25) .. controls (313.95,184.59) and (313.41,184.05) .. (312.75,184.05) .. controls (312.08,184.05) and (311.55,184.59) .. (311.55,185.25) .. controls (311.55,185.91) and (312.08,186.45) .. (312.75,186.45) .. controls (313.41,186.45) and (313.95,185.91) .. (313.95,185.25) -- cycle ;
\draw  [fill={rgb, 255:red, 0; green, 0; blue, 0 }  ,fill opacity=1 ] (324.07,185.2) .. controls (324.07,184.54) and (323.53,184) .. (322.87,184) .. controls (322.21,184) and (321.67,184.54) .. (321.67,185.2) .. controls (321.67,185.86) and (322.21,186.4) .. (322.87,186.4) .. controls (323.53,186.4) and (324.07,185.86) .. (324.07,185.2) -- cycle ;
\draw  [fill={rgb, 255:red, 0; green, 0; blue, 0 }  ,fill opacity=1 ] (334.08,185.25) .. controls (334.08,184.59) and (333.55,184.05) .. (332.88,184.05) .. controls (332.22,184.05) and (331.68,184.59) .. (331.68,185.25) .. controls (331.68,185.91) and (332.22,186.45) .. (332.88,186.45) .. controls (333.55,186.45) and (334.08,185.91) .. (334.08,185.25) -- cycle ;
\draw  [fill={rgb, 255:red, 0; green, 0; blue, 0 }  ,fill opacity=1 ] (271.55,185.14) .. controls (271.55,182.38) and (273.78,180.14) .. (276.55,180.14) .. controls (279.31,180.14) and (281.55,182.38) .. (281.55,185.14) .. controls (281.55,187.9) and (279.31,190.14) .. (276.55,190.14) .. controls (273.78,190.14) and (271.55,187.9) .. (271.55,185.14) -- cycle ;
\draw  [fill={rgb, 255:red, 0; green, 0; blue, 0 }  ,fill opacity=1 ] (271.55,245) .. controls (271.55,242.24) and (273.78,240) .. (276.55,240) .. controls (279.31,240) and (281.55,242.24) .. (281.55,245) .. controls (281.55,247.76) and (279.31,250) .. (276.55,250) .. controls (273.78,250) and (271.55,247.76) .. (271.55,245) -- cycle ;
\draw    (276.55,185.14) -- (276.55,245) ;
\draw  [dash pattern={on 4.5pt off 4.5pt}]  (276.55,185.14) -- (226.71,245) ;
\draw  [dash pattern={on 4.5pt off 4.5pt}]  (226.71,185.14) -- (276.55,245) ;
\draw  [fill={rgb, 255:red, 0; green, 0; blue, 0 }  ,fill opacity=1 ] (314.11,245.25) .. controls (314.11,244.59) and (313.58,244.05) .. (312.91,244.05) .. controls (312.25,244.05) and (311.71,244.59) .. (311.71,245.25) .. controls (311.71,245.91) and (312.25,246.45) .. (312.91,246.45) .. controls (313.58,246.45) and (314.11,245.91) .. (314.11,245.25) -- cycle ;
\draw  [fill={rgb, 255:red, 0; green, 0; blue, 0 }  ,fill opacity=1 ] (324.24,245.2) .. controls (324.24,244.54) and (323.7,244) .. (323.04,244) .. controls (322.37,244) and (321.84,244.54) .. (321.84,245.2) .. controls (321.84,245.86) and (322.37,246.4) .. (323.04,246.4) .. controls (323.7,246.4) and (324.24,245.86) .. (324.24,245.2) -- cycle ;
\draw  [fill={rgb, 255:red, 0; green, 0; blue, 0 }  ,fill opacity=1 ] (334.25,245.25) .. controls (334.25,244.59) and (333.72,244.05) .. (333.05,244.05) .. controls (332.39,244.05) and (331.85,244.59) .. (331.85,245.25) .. controls (331.85,245.91) and (332.39,246.45) .. (333.05,246.45) .. controls (333.72,246.45) and (334.25,245.91) .. (334.25,245.25) -- cycle ;

\draw (272.71,42.4) node [anchor=north west][inner sep=0.75pt]    {$x$};
\draw (141,118.83) node [anchor=north west][inner sep=0.75pt]    {$N( x)$};
\draw (141,175.63) node [anchor=north west][inner sep=0.75pt]    {$N^{2}( x)$};
\draw (141,235.63) node [anchor=north west][inner sep=0.75pt]    {$N^{3}( x)$};
\draw (200.81,178.29) node [anchor=north west][inner sep=0.75pt]    {$v^{*}$};
\draw (198.24,237.54) node [anchor=north west][inner sep=0.75pt]    {$w^{*}$};
\draw (219.91,117.81) node [anchor=north west][inner sep=0.75pt]    {$u^{*}$};
\draw (285.6,240.4) node [anchor=north west][inner sep=0.75pt]    {$w$};

\end{tikzpicture}}
\end{center}
Observe that $d(w^*,w)\in \{1,3\}$ and that if $d(w^*,w)=3$, it must hold that $d(w^*)=1$ and $d(v^*,w)=2$. We consider the line $l_1=\li{wv^*}$, observing that $w^*\in l_1$. Since $x\notin l_1$ and $l_1\cap N^3(w^*)\cap N(x) =\emptyset$, the line $l_1$ does not belong to $\Li^x_2\cup \Li^x_3\cup L_{A_2}\cup L_1^*$. It follows that $\ell(G)=n-1$ and $\Li(G)=\Li^x_2\cup \Li^x_3\cup L_{A_2}\cup L_1^*\cup \{l_1\}$. Consider $v\in N(w)\cap N^2(w^*)\cap N^2(x)\neq \emptyset$. As $x,v^*,w^*\in \li{xu^*}$, we have that $\li{xu^*}\in \{\li{xv^*}\}\cup L_{A_2}$. If $\li{xu^*}=\li{w'u'}\in L_{A_2}$, necessarily $d(w',u^*)=2$, thus $u'=u^*$, and since $w^*\in \li{xu^*}=\li{w'u^*}$, the vertex $w'$ must be equal to $w^*$. This implies that $\li{xu^*}\in \{\li{xv^*},\li{w^*u^*}\}$. Observe that $v\notin \li{xv^*}\cup \li{w^*u^*}$, therefore $v\notin \li{xu^*}$ and $d(v,u^*)=2$. We now consider the line $l_2=\li{vu^*}$. Given that $x,w^*\notin  l_2$, $l_2\notin \Li(G)$, which is a contradiction. 
\end{proof}

We have that $|N^3(x)|=1$, hence $N^3(x)=\{w^*\}$, $\Li^x_3=\{\li{xw^*}\}$ and $L_{A_2}=\{\li{w^*u}:u\in N^2(w^*)\cap N(x)\}$. Furthermore, $w^*$ has degree 1, which implies that $w^*\notin R$, that is, $\Li^{w^*}_2\cap \Li^{w^*}_3=\emptyset$. Now consider the set of lines $L_2^*=\{\li{w^*v}:v\in N^2(x)-v^*\}$. The lines in $L_2^*$ do not contain $x$, thus $L_2^*$ is disjoint from $\Li^x_2,\Li^x_3$ and $ L_{A_2}$. Given that $\Li^{w^*}_2\cap \Li^{w^*}_3=\emptyset$, the set $L_2^*$ is disjoint from $L_1^*$ and $|L_2^*|=|N^2(x)|-1$. The set $\Li^x_2\cup  \Li^x_3 \cup L_{A_2} \cup L_1^*\cup L_2^*$ contains $n-3+|N^2(x)|$ lines, implying that $|N^2(x)|\leq 2$.

\begin{claim}
$|N^2(x)|=1$.
\end{claim}

\begin{proof}
By contradiction, suppose that $|N^2(x)|=2$, and let $N^2(x)=\{v^*,v\}$. We have that $\ell(G)=n-1$ and $\Li(G)=\Li^x_2\cup  \Li^x_3\cup L_{A_2}\cup L_1^*\cup L_2^*$, where $L_2^*=\{\li{w^*v}\}$. We note that $V\notin L_2^*\cup L_1^*$, and since $|N^2(x)|=2$, $V\notin \Li^x_2$. Moreover, as $\Li^x_3=\{\li{xv^*}\}$, we also have that $V\notin  \Li^x_3$. Given that $|N^3(x)|=1$, the line $\li{w^*v^*}$ equals to $V$, which implies that $V\in L_{A_2}$. It follows that $N^2(w^*)\cap N(x)=\{u^*\}$, $\li{w^*u^*}=V$, $v\in N^3(w^*)\cap N(u^*)$ and $N(x)-u^*\subseteq N^3(w^*)\cap N(u^*)$. Let $u\in N(x)-u^*$, hence $u\in N^3(w^*)\cap N(u^*)$. We now consider the line $l=\li{xu}$ and observe that $w^*\notin l$, therefore it must hold that $l=\li{xv}\in \Li ^x_2$.  The vertex $u^*$ belongs to $\li{xv}\setminus \li{xu}$, leading to a contradiction.
\end{proof}

We now have that $N^2(x)=\{v^*\}$, therefore $\Li^x_2=\{\li{xv^*}\}=\{\li{xw^*}\}=\Li^x_3$ and $L_2^*=\emptyset$. The set $\Li^x_2\cup  \Li^x_3\cup L_{A_2}\cup L_1^*$ has $n-2$ lines, and all of them contain $w^*$. Consider the set of lines $L=\{\li{uu'}:u,u'\in N^2(w^*)\cap N(x)\}\cup \{\li{xu}:u\in N^3(w^*)\cap N(x)\}$. As no line in $L$ contains $w^*$, the set $L$ is disjoint from $\Li^x_2, \Li^x_3,L_{A_2}$ and $L_1^*$. Since $|N(x)|\geq 2$, we obtain that $|L|\geq 1$. It follows that $|L|=1$, $\ell(G)=n-1$ and $\Li(G)= \{\li{xw^*}\}\cup L_{A_2}\cup L_1^*\cup L$. Observe that the line $V=\li{w^*v^*}$ must belong to $\{\li{xw^*}\}\cup L_{A_2}$.

\begin{claim}
If $V=\li{xw^*}$, then $G$ is isomorphic to $G_a$ or $G_b$.
\end{claim}

\begin{proof}
Assume that $V=\li{xw^*}$. In this case $N(x)\subseteq N^2(w^*)$ and $L=\{\li{uu'}:u,u'\in  N(x)\}$. Given that $|L|=1$, let $L=\{l\}$. If $x\in l$, the graph $G[N(x)]$ has no edges, and if $x\notin l$, $G[N(x)]$ must be a complete graph. In both cases, we conclude that $|N(x)|\leq 2$, therefore $|N(x)|=2$. Let $N(x)=\{u,u^*\}$, hence $V=\{x,u,u^*,v^*,w^*\}$. If $uu^*\notin E$, then $G$ is isomorphic to $G_a$, and if $uu^*\in E$, then $G$ is isomorphic to $G_b$.   
\end{proof}

\begin{claim}
If $V\in L_{A_2}$, then $G$ is isomorphic to $G_c$.
\end{claim}

\begin{proof}
Assume that $V\in L_{A_2}$. In this case $N^2(w^*)\cap N(x)=\{u^*\}$, $V=\li{w^*u^*}$ and $L= \{\li{xu}:u\in N(x)-u^*\}$. Furthermore, $ N(x)-u^*\subseteq N(u^*)$, therefore no line in $L$ contains $u^*$. We will now prove that $|N(x)|=2$. By contradiction, suppose that $|N(x)|\geq 3$ and consider distinct $u,u'\in N(x)-u^*=N^3(w^*)\cap N(x)$. We consider the line $l=\li{uu'}$, and note that $w^*\notin l$, thus $l\in L$ and $x\in l$. If follows that $d(u,u')=2$ and $u^*\in l$, yielding a contradiction. We conclude that $|N(x)|=2$, and let $N(x)=\{u,u^*\}$, hence $V=\{x,u,u^*,v^*,w^*\}$ and $G$ is isomorphic to $G_c$. 
\end{proof}

This concludes the proof of Proposition \ref{PropR1A1empty}.

\section{Proof of the main theorem: Case 2}\label{sec:Rempty}
In this section we assume that $R$ is empty, that is, $\forall \,x\in V$, $\Li^x_2\cap \Li^x_3=\emptyset$. We shall prove the following Proposition.

\begin{prop}
$G$ is isomorphic to $M'_{6}$ or $M'_{8}$.
\label{PropRempty}
\end{prop}

\subsubsection*{Proof of Proposition \ref{PropRempty}}

We begin by defining sets that will be useful in describing the graph $G$. For $x,y\in V$ such that $d(x,y)=3$, we define $A(x,y)=\{z\in N^2(y)\cap N(x): \li{yz}\notin \Li^x_2\}$, $B(x,y)=\{z\in N^2(y)\cap N(x): \li{yz}\in \Li^x_2\}$ and $C(x,y)=N^3(y)\cap N(x)$. For this section, fix $x,w\in V$ such that $d(x,w)=3$.  We define  $D=N^2(x)\setminus N(w)$ and $D'=N^3(x)\setminus (N(w)\cup \{w\})$. Observe that $N(x)$ is the disjoint union of  $A(x,w),B(x,w)$ and $C(x,w)$, that $N^2(x)$ is the disjoint union of $A(w,x),B(w,x)$ and $D$, and that $N^3(x)$ is the disjoint union of $\{w\}, C(w,x)$ and $D'$. The diagram below illustrates the graph $G$.

\begin{figure}[H]
\begin{center}      
\scalebox{.75}{\tikzset{every picture/.style={line width=0.75pt}} 

\begin{tikzpicture}[x=0.75pt,y=0.75pt,yscale=-1,xscale=1]

\draw  [fill={rgb, 255:red, 0; green, 0; blue, 0 }  ,fill opacity=1 ] (220,55) .. controls (220,52.24) and (222.24,50) .. (225,50) .. controls (227.76,50) and (230,52.24) .. (230,55) .. controls (230,57.76) and (227.76,60) .. (225,60) .. controls (222.24,60) and (220,57.76) .. (220,55) -- cycle ;
\draw   (70,120) .. controls (70,103.43) and (139.4,90) .. (225,90) .. controls (310.6,90) and (380,103.43) .. (380,120) .. controls (380,136.57) and (310.6,150) .. (225,150) .. controls (139.4,150) and (70,136.57) .. (70,120) -- cycle ;
\draw    (170.23,91.65) .. controls (180.23,104.32) and (180.23,136.32) .. (170.23,148.32) ;
\draw    (279.68,91.65) .. controls (289.68,104.32) and (289.68,136.32) .. (279.68,148.32) ;
\draw   (70,210) .. controls (70,193.43) and (139.4,180) .. (225,180) .. controls (310.6,180) and (380,193.43) .. (380,210) .. controls (380,226.57) and (310.6,240) .. (225,240) .. controls (139.4,240) and (70,226.57) .. (70,210) -- cycle ;
\draw    (170.23,181.65) .. controls (180.23,194.32) and (180.23,226.32) .. (170.23,238.32) ;
\draw    (279.68,181.65) .. controls (289.68,194.32) and (289.68,226.32) .. (279.68,238.32) ;
\draw   (70,300) .. controls (70,283.43) and (139.4,270) .. (225,270) .. controls (310.6,270) and (380,283.43) .. (380,300) .. controls (380,316.57) and (310.6,330) .. (225,330) .. controls (139.4,330) and (70,316.57) .. (70,300) -- cycle ;
\draw    (170.23,271.65) .. controls (180.23,284.32) and (180.23,316.32) .. (170.23,328.32) ;
\draw    (279.68,271.65) .. controls (289.68,284.32) and (289.68,316.32) .. (279.68,328.32) ;
\draw  [fill={rgb, 255:red, 0; green, 0; blue, 0 }  ,fill opacity=1 ] (130.4,299.03) .. controls (130.4,296.27) and (132.64,294.03) .. (135.4,294.03) .. controls (138.16,294.03) and (140.4,296.27) .. (140.4,299.03) .. controls (140.4,301.79) and (138.16,304.03) .. (135.4,304.03) .. controls (132.64,304.03) and (130.4,301.79) .. (130.4,299.03) -- cycle ;

\draw (221,32.4) node [anchor=north west][inner sep=0.75pt]    {$x$};
\draw (7,110) node [anchor=north west][inner sep=0.75pt]    {$N( x)$};
\draw (6.8,194) node [anchor=north west][inner sep=0.75pt]    {$N^{2}( x)$};
\draw (6.4,274.4) node [anchor=north west][inner sep=0.75pt]    {$N^{3}( x)$};
\draw (111,111.4) node [anchor=north west][inner sep=0.75pt]    {$A( x,w)$};
\draw (111,201.4) node [anchor=north west][inner sep=0.75pt]    {$A( w,x)$};
\draw (110.6,299.23) node [anchor=north west][inner sep=0.75pt]    {$w$};
\draw (197,112.4) node [anchor=north west][inner sep=0.75pt]    {$B( x,w)$};
\draw (197,202.4) node [anchor=north west][inner sep=0.75pt]    {$B( w,x)$};
\draw (297,112.4) node [anchor=north west][inner sep=0.75pt]    {$C( x,w)$};
\draw (197.26,290.51) node [anchor=north west][inner sep=0.75pt]    {$C( w,x)$};
\draw (306,202.4) node [anchor=north west][inner sep=0.75pt]    {$D$};
\draw (306,292.4) node [anchor=north west][inner sep=0.75pt]    {$D'$};
\draw (420,90) node [anchor=north west][inner sep=0.75pt] {\large $ \begin{array}{l}
A(x,w) =\{u\in N^2(w) \cap N(x): \li{wu}\notin \Li^x_2\} \vspace{3pt} \\
B(x,w)=\{u\in N^2(w) \cap N(x): \li{wu}\in \Li^x_2\} \vspace{3pt} \\
C(x,w)=N^3(w)\cap N(x)\vspace{3pt} \\
A(w,x)=\{v\in N^2(x)\cap N(w): \li{xv}\notin \Li^w_2\}\vspace{3pt} \\
B(w,x)=\{v\in N^2(x)\cap N(w): \li{xv}\in \Li^w_2\}\vspace{3pt} \\
C(w,x)=N^3(x)\cap N(w)\vspace{3pt} \\
D=N^2(x)\setminus N(w)\vspace{3pt} \\
D'=N^3(x)\setminus \left ( N(w)\cup \{w\}\right)
\end{array}$};

\end{tikzpicture}}
\end{center}
\caption{Diagram of $G$.}
\label{fig:RemptyDiagram}
\end{figure}
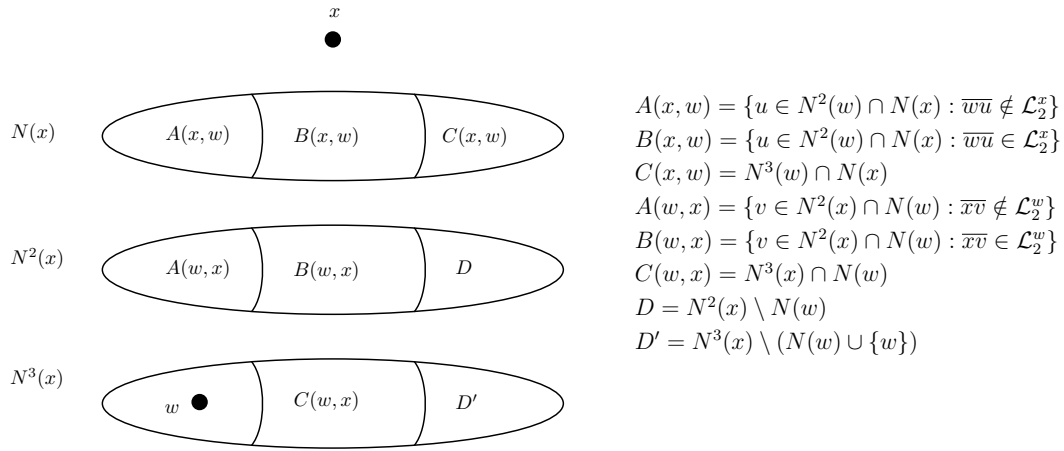

\begin{claim}
If $u\in B(x,w)$ and $v\in N^2(x)$ such that $\li{wu}=\li{xv}$. Then $v\in B(w,x)\cap N(u)$ and $\left ( A(x,w)\cup B(x,w)-u \right )\subseteq N^2(v)$.
\label{Claimwuxv}
\end{claim}        

\begin{proof}
Let $u\in B(x,w)$ and $v\in N^2(x)$ such that $\li{wu}=\li{xv}$, in particular, $\li{xv}\in \Li^w_2$. Consider $v'\in N(w)\cap N(u)$, and observe that $v'\in N^2(x)$. Since $v'\in \li{wu}=\li{xv}$, we have that $v'=v$ and $v\in B(w,x)\cap N(u)$. To prove the second statement, consider $u'\in A(x,w)\cup B(x,w)-u$. If $d(v,u')\in \{1,3\}$, we obtain the contradiction $u'\in \li{xv}=\li{wu}$, concluding that $u'\in N^2(v)$.  
\end{proof}

\begin{claim}
The function  $f: B(x,w)\to B(w,x)$ given by $f(u)=v\in N^2(x)$ such that $\li{wu}=\li{xv}$ is well defined and bijective. Moreover, $E(B(x,w),A(w,x)\cup B(w,x))=E(B(w,x),A(x,w)\cup B(x,w))=\{uf(u): u \in B(x,w)\}$.
\label{Claimfuncmatch}
\end{claim}

\begin{proof}
Let $u\in B(x,w)$, by definition, there exists $ v\in N^2 (x)$ such that $\li{wu}=\li{xv}$. It is immediate that this $v$ is unique, and by Claim \ref{Claimwuxv}, the vertex $v$ belongs to $B(w,x)$, hence $f$ is well defined. We now show that $f$ is surjective. Consider $v\in B(w,x)$ and $u\in N(v)\cap N(x)$. As $u\in \li{xv}\in \Li^w_2$, we have that $\li{wu}=\li{xv}\in \Li^x_2$, therefore $u\in  B(x,w)$ and $v=f(u)$. The injectivity of $f$ follows directly from the fact that $\li{wu}\neq \li{wu'}$ for distinct $u,u'\in B(x,w)$. We now prove the second statement. By Claim \ref{Claimwuxv}, $E(B(w,x),A(x,w)\cup B(x,w))=\{uf(u): u \in B(x,w)\}$. In particular, $E(A(x,w),B(w,x))=\emptyset$, hence reversing the roles of $x$ and $w$ we obtain that $E(A(w,x),B(x,w))=\emptyset$, concluding that $E(B(x,w),A(w,x)\cup B(w,x))=\{uf(u): u \in B(x,w)\}$.
\end{proof}

The previous two claims imply that for $v\in B(w,x)$, we have that $A(x,w)\subseteq N^2(v)$. Thus, for $u\in A(x,w)$, it holds that $B(w,x)\subseteq N^2(u)$. Reversing the roles of $x$ and $w$, we conclude that for $v\in A(w,x)$, the inclusion $B(x,w)\subseteq N^2(v)$ holds. 

\begin{claim}
$E(B(x,w),C(x,w))=E(B(w,x),C(w,x))=E(B(w,x),C(x,w))=E(B(w,x),D')=\emptyset$.
\end{claim}

\begin{proof}
It is immediate that $E(B(w,x),C(x,w))=\emptyset$. We first show that $E(B(x,w),C(x,w))=\emptyset$ by contradiction. If $uu'\in E(B(x,w),C(x,w))$, considering $v=f(u)$, we have that $d(v,u')=2$, therefore $u'\in \li{wu}\setminus \li{xv}$, which is a contradiction. Reversing the roles of $x$ and $w$, we conclude that $E(B(w,x),C(w,x))=\emptyset$. For the last equality, if $vw'\in E(B(w,x),D')$, consider $u=f^{-1}(v)$, and observe that $d(w',u)=2$, hence $w'\in \li{xv}\setminus \li{wu}$, leading to a contradiction.
\end{proof}

We will gradually prove that $A(x,w)=A(w,x)=C(x,w)=C(w,x)=D=D'=\emptyset$. Consider the sets of lines $\Li^x_2, \Li^x_3, L_A=\{\li{wu}:u\in A(x,w)\},L_C=\{\li{wu}:u\in C(x,w)\},L_D=\{\li{wv}:v\in D\}$ and $L_{D'}=\{\li{ww'}:w'\in D'\}$. We observe that all the lines in $\Li^x_2\cup \Li^x_3 \cup L_A$ contain $x$, and that no line in  $L_C\cup L_D\cup L_{D'}$ contains $x$.

\begin{claim}
The sets $\Li^x_2, \Li^x_3, L_A,L_C,L_D$ and $L_{D'}$ are pairwise disjoint.
\label{ClaimDisjoint1}
\end{claim}

\begin{proof}
By hypothesis, $\Li^x_2$ is disjoint from $\Li^x_3$, and by definition of $A(x,w)$, $\Li^x_2$ is disjoint from $L_A$. As no line in $L_C\cup L_D\cup L_{D'}$ contains $x$, the sets $\Li^x_2,\Li^x_3$ and $L_A$ are disjoint from $L_C,L_D$ and $L_{D'}$. Since $\Li^w_2\cap \Li^w_3=\emptyset$ and $L_C\cup L_D\cup L_{D'}\subseteq \Li^w_2\cup \Li^w_3 $, the sets $L_C,L_D$ and $L_{D'}$ are pairwise disjoint. It remains to show that $\Li^x_3\cap L_A=\emptyset$. By contradiction, suppose that $\li{wu}\in L_A\cap \Li^x_3$, thus $\li{wu}=\li{xw}$. We have that the line $\li{wu}=\li{wx}$ belongs to $\Li^w_2\cap \Li^w_3=\emptyset$, a contradiction. 
\end{proof}

The set $\Li^x_2\cup  \Li^x_3\cup  L_A\cup L_C\cup L_D\cup L_{D'}$ contains
$$
|N^2(x)|+ |N^3(x)|+ |A(x,w)|+ |C(x,w)|+ |D|+ |D'|
$$
lines.

\begin{claim}
$B(x,w)\neq \emptyset$.
\end{claim}

\begin{proof}
For the sake of contradiction, suppose that $B(x,w)=\emptyset$. It follows that $A(x,w)\cup C(x,w)=N(x)$, hence $\ell(G)=n-1$, $\Li(G)=\Li^x_2\cup  \Li^x_3\cup  L_A\cup L_C$ and $D=D'=\emptyset$. If $|N^2(x)|\geq 2$, consider a line $l=\li{vv'}$ with distinct $v,v'\in N^2(x)$. Since $D=\emptyset$, the vertices $v$ and $v'$ must belong to $N(w)$. If $vv'\notin E$, we have that $l\cap C(x,w)=\emptyset$, and if $vv'\in E$, we have that $w\notin l$. In both cases we conclude that $l\notin L_C$. As $x\notin l$, we obtain that $l\notin \Li(G)$, which is a contradiction. Therefore, we conclude that $|N^2(x)|=1$ and let $N^2(x)=\{v\}$. Observe that $v\in N(w)$, $A(x,w)\subseteq N(v)$ and  $C(x,w)\subseteq N^2(v)$, hence $\li{xv}\cap \left ( \{x\}\cup N(x)\cup N^2(x)\cup \{w\}\right )=\{x,v,w\}\cup A(x,w)=\li{xw}\cap \left ( \{x\}\cup N(x)\cup N^2(x)\cup \{w\}\right )$. Given that $\li{xv}\neq \li{xw}$, there must exist a vertex $w'$ in $N^3(x)-w$ such that $w'\in \li{xv}\setminus \li{xw}$, and necessarily $w'v\in E$. We now consider the line $l=\li{ww'}$. As $x\notin l$, the line $l$ does not belong to $ \Li^x_2\cup \Li^x_3\cup L_A$. Since $N^2(x)=\{v\}$, we have that $C(x,w)\subseteq N^3(w')$, therefore $l\cap C(x,w)=\emptyset$ and $l\notin L_C$. If follows that $l\notin \Li(G)$, yielding a contradiction.
\end{proof}

Fix $u^*\in B(x,w)$ and $v^*=f(u^*)$, thus, $v^*\in B(w,x)$ and $\li{wu^*}=\li{xv^*}$. $G$ is depicted in the following diagram.

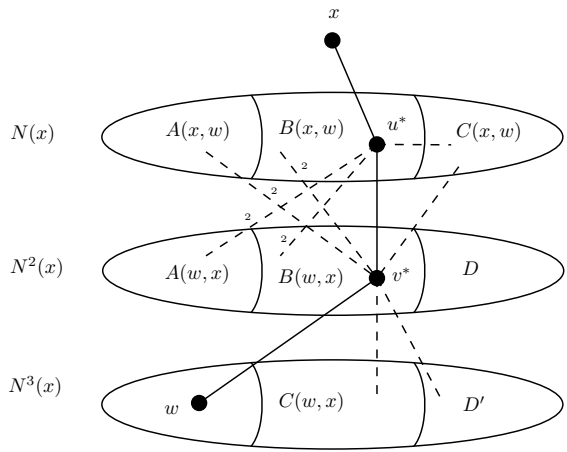
\begin{figure}[H]
\begin{center}      
\scalebox{.75}{\tikzset{every picture/.style={line width=0.75pt}} 

\begin{tikzpicture}[x=0.75pt,y=0.75pt,yscale=-1,xscale=1]

\draw  [fill={rgb, 255:red, 0; green, 0; blue, 0 }  ,fill opacity=1 ] (310,25) .. controls (310,22.24) and (312.24,20) .. (315,20) .. controls (317.76,20) and (320,22.24) .. (320,25) .. controls (320,27.76) and (317.76,30) .. (315,30) .. controls (312.24,30) and (310,27.76) .. (310,25) -- cycle ;
\draw   (160,90) .. controls (160,73.43) and (229.4,60) .. (315,60) .. controls (400.6,60) and (470,73.43) .. (470,90) .. controls (470,106.57) and (400.6,120) .. (315,120) .. controls (229.4,120) and (160,106.57) .. (160,90) -- cycle ;
\draw    (260.23,61.65) .. controls (270.23,74.32) and (270.23,106.32) .. (260.23,118.32) ;
\draw    (369.68,61.65) .. controls (379.68,74.32) and (379.68,106.32) .. (369.68,118.32) ;
\draw   (160.57,180) .. controls (160.57,163.43) and (229.97,150) .. (315.57,150) .. controls (401.18,150) and (470.57,163.43) .. (470.57,180) .. controls (470.57,196.57) and (401.18,210) .. (315.57,210) .. controls (229.97,210) and (160.57,196.57) .. (160.57,180) -- cycle ;
\draw    (260.23,151.65) .. controls (270.23,164.32) and (270.23,196.32) .. (260.23,208.32) ;
\draw    (369.68,151.65) .. controls (379.68,164.32) and (379.68,196.32) .. (369.68,208.32) ;
\draw   (160,270) .. controls (160,253.43) and (229.4,240) .. (315,240) .. controls (400.6,240) and (470,253.43) .. (470,270) .. controls (470,286.57) and (400.6,300) .. (315,300) .. controls (229.4,300) and (160,286.57) .. (160,270) -- cycle ;
\draw    (260.23,241.65) .. controls (270.23,254.32) and (270.23,286.32) .. (260.23,298.32) ;
\draw    (369.68,241.65) .. controls (379.68,254.32) and (379.68,286.32) .. (369.68,298.32) ;
\draw  [fill={rgb, 255:red, 0; green, 0; blue, 0 }  ,fill opacity=1 ] (220.4,269.03) .. controls (220.4,266.27) and (222.64,264.03) .. (225.4,264.03) .. controls (228.16,264.03) and (230.4,266.27) .. (230.4,269.03) .. controls (230.4,271.79) and (228.16,274.03) .. (225.4,274.03) .. controls (222.64,274.03) and (220.4,271.79) .. (220.4,269.03) -- cycle ;
\draw  [fill={rgb, 255:red, 0; green, 0; blue, 0 }  ,fill opacity=1 ] (340,95) .. controls (340,92.24) and (342.24,90) .. (345,90) .. controls (347.76,90) and (350,92.24) .. (350,95) .. controls (350,97.76) and (347.76,100) .. (345,100) .. controls (342.24,100) and (340,97.76) .. (340,95) -- cycle ;
\draw  [fill={rgb, 255:red, 0; green, 0; blue, 0 }  ,fill opacity=1 ] (340,185) .. controls (340,182.24) and (342.24,180) .. (345,180) .. controls (347.76,180) and (350,182.24) .. (350,185) .. controls (350,187.76) and (347.76,190) .. (345,190) .. controls (342.24,190) and (340,187.76) .. (340,185) -- cycle ;
\draw    (345,95) -- (345,185) ;
\draw    (345,185) -- (225.4,269.03) ;
\draw    (315,25) -- (345,95) ;
\draw  [dash pattern={on 4.5pt off 4.5pt}]  (280,100) -- (345,185) ;
\draw  [dash pattern={on 4.5pt off 4.5pt}]  (230,100) -- (345,185) ;
\draw  [dash pattern={on 4.5pt off 4.5pt}]  (345,95) -- (394.88,95.29) ;
\draw  [dash pattern={on 4.5pt off 4.5pt}]  (345,185) -- (389.6,269.1) ;
\draw  [dash pattern={on 4.5pt off 4.5pt}]  (345,185) -- (345,267.19) ;
\draw  [dash pattern={on 4.5pt off 4.5pt}]  (400,110) -- (345,185) ;
\draw  [dash pattern={on 4.5pt off 4.5pt}]  (345,95) -- (280,170) ;
\draw  [dash pattern={on 4.5pt off 4.5pt}]  (345,95) -- (230,170) ;

\draw (311,2.4) node [anchor=north west][inner sep=0.75pt]    {$x$};
\draw (97,80) node [anchor=north west][inner sep=0.75pt]    {$N( x)$};
\draw (96.8,169.11) node [anchor=north west][inner sep=0.75pt]    {$N^{2}( x)$};
\draw (95.73,249.51) node [anchor=north west][inner sep=0.75pt]    {$N^{3}( x)$};
\draw (201.29,78.13) node [anchor=north west][inner sep=0.75pt]    {$A( x,w)$};
\draw (200.78,173.75) node [anchor=north west][inner sep=0.75pt]    {$A( w,x)$};
\draw (200.6,269.23) node [anchor=north west][inner sep=0.75pt]    {$w$};
\draw (277.25,76.76) node [anchor=north west][inner sep=0.75pt]    {$B( x,w)$};
\draw (276.94,176.54) node [anchor=north west][inner sep=0.75pt]    {$B( w,x)$};
\draw (396.67,78.84) node [anchor=north west][inner sep=0.75pt]    {$C( x,w)$};
\draw (277.4,260.01) node [anchor=north west][inner sep=0.75pt]    {$C( w,x)$};
\draw (400.89,172.18) node [anchor=north west][inner sep=0.75pt]    {$D$};
\draw (401.33,262.84) node [anchor=north west][inner sep=0.75pt]    {$D'$};
\draw (353.61,176.12) node [anchor=north west][inner sep=0.75pt]    {$v^{*}$};
\draw (350.33,74.62) node [anchor=north west][inner sep=0.75pt]    {$u^{*}$};
\draw (292.44,105.9) node [anchor=north west][inner sep=0.75pt]  [font=\tiny]  {$2$};
\draw (278.67,154.57) node [anchor=north west][inner sep=0.75pt]  [font=\tiny]  {$2$};
\draw (272.06,122.68) node [anchor=north west][inner sep=0.75pt]  [font=\tiny]  {$2$};
\draw (254.5,140.23) node [anchor=north west][inner sep=0.75pt]  [font=\tiny]  {$2$};

\end{tikzpicture}}
\end{center}
\caption{Diagram of $G$. Dashed lines represent missing edges. Numbers indicate distances between the corresponding vertices.}  
\label{fig:RemptyDiagram2}
\end{figure}

We consider the set of lines $L_B=\{\li{v^*u}:u\in  A(x,w)\cup B(x,w)-u^*\}$. Observe that $\forall \, l\in L_B$, it holds that $x,w\notin l$.

\begin{claim}
The sets $\Li^x_2, \Li^x_3, L_A,L_C,L_D,L_{D'}$ and $L_B$ are pairwise disjoint.
\end{claim}

\begin{proof}
By Claim \ref{ClaimDisjoint1}, it suffices to prove that the set $L_B$ is disjoint from $\Li^x_2, \Li^x_3, L_A,L_C,L_D$ and $L_{D'}$. This is immediate from the fact that the lines in $L_B$ contain neither $x$ nor $w$.
\end{proof}

The set $\Li^x_2\cup \Li^x_3\cup  L_A\cup L_B \cup L_C \cup L_D\cup L_{D'}$ contains  
$$
|N^2(x)|+ |N^3(x)|+ |A(x,w)|+|A(x,w)|+|B(x,w)|-1 + |C(x,w)|+ |D|+ |D'| =n-2+|A(x,w)|+|D|+|D'|
$$
lines. Observe that $|A(x,w)|+|D|+|D'|\leq 1$.

\begin{claim}
$A(x,w)\neq \emptyset \Longleftrightarrow A(w,x)\neq \emptyset$.
\label{ClaimAxwAwx}
\end{claim}

\begin{proof}
Assume that $A(x,w)\neq \emptyset$. Let $u\in A(x,w)$ and $v\in N(w)\cap N(u)$. By Claim \ref{Claimfuncmatch} it holds that $v\in A(w,x)$. Reversing the roles of $x$ and $w$ we obtain the converse implication.
\end{proof}

\begin{claim}
$A(x,w)\neq\emptyset\Longrightarrow |A(x,w)|=1\,\, \land \,\, C(x,w)=\emptyset$.
\label{ClaimAxwCxw}
\end{claim}

\begin{proof}
We assume that $A(x,w)\neq \emptyset$. It follows that $|A(x,w)|=1$, $\ell(G)=n-1$, $D=D'=\emptyset$ and $\Li(G)=\Li^x_2\cup \Li^x_3\cup  L_A\cup L_B \cup L_C$. We will now prove that $C(x,w)=\emptyset$ by contradiction. Suppose that $C(x,w)\neq \emptyset$ and consider the line $l=\li{v^*u}$ with $u\in C(x,w)$. Since $\Li^{v^*}_2\cap \Li^{v^*}_3=\emptyset$ and $v^*u\notin E$, the line $l$ does not belong to $L_B$. If $d(v^*,u)=3$, since $N(v^*)\cap N^3(x)=\{w\}$, it must hold that $l\cap N^3(x)=\emptyset$ and therefore $l\notin \Li^x_3\cup L_A\cup L_C$. Given that $w\in \li{xv^*}\setminus l$ and $v^*\in l$, we have that $l\notin \Li^x_2$, hence $l\notin \Li(G)$. This is a contradiction, thus we conclude that $d(v^*,u)=2$. This implies that $x\notin l$, therefore $\li{v^*u}=l=\li{wu}\in L_C$. The line $l$ belongs to $\Li^u_2\cap \Li ^u_3=\emptyset$, which is a contradiction. 
\end{proof}

\begin{claim}
$A(x,w)=\emptyset$.
\end{claim}

\begin{proof}
For the sake of contradiction, we suppose that $A(x,w)\neq \emptyset$. By Claim \ref{ClaimAxwCxw} we obtain that $|A(x,w)|=1$ and $C(x,w)=\emptyset$. By Claim \ref{ClaimAxwAwx}, we also have that $A(w,x)\neq \emptyset$. Reversing the roles of $x$ and $w$ in Claim \ref{ClaimAxwCxw}, we obtain that $|A(w,x)|=1$ and $C(w,x)=\emptyset$. Since $A(x,w)\neq \emptyset$, we also have that $D=D'=\emptyset$. Consider $\hat{u}\in A(x,w)$ and $\hat{v}\in N(w)\cap N(\hat{u})$. By Claim \ref{Claimfuncmatch}, the vertex $\hat{v}$ must belong to $ A(w,x)$. Recall that $B(x,w)\subseteq N^2(\hat{v})$ and $B(w,x)\subseteq N^2(\hat{u})$. As $C(x,w)=C(w,x)=D=D'=\emptyset$ and $|A(x,w)|=|A(w,x)|=1$, we have that $\li{w\hat{u}}=\{x,\hat{u},\hat{v},w\}=\li{x\hat{v}}\in\Li^x_2$, leading to a contradiction.
\end{proof}

By Claim \ref{ClaimAxwAwx}, we also conclude that $A(w,x)=\emptyset$. The set $\Li^x_2\cup \Li^x_3\cup L_B \cup L_C\cup L_D\cup L_{D'}$ contains 
$$
|N^2(x)|+ |N^3(x)|+|B(x,w)|-1 + |C(x,w)|+ |D|+ |D'| =n-2+|D|+|D'|
$$
lines. Since $E(B(w,x),D')=\emptyset$, we have that $D'\neq \emptyset \Longrightarrow D\neq \emptyset \Longrightarrow \ell(G)\geq n$. It follows that $D'=\emptyset$.   

\begin{claim}
 $D=\emptyset$.
\end{claim}

\begin{proof}
By contradiction, suppose that $D\neq \emptyset$. It follows that $|D|=1$, $\ell(G)=n-1$ and $\Li(G)=\Li^x_2\cup  \Li^x_3\cup L_B \cup L_C \cup L_D$. Let $D=\{v\}$. We will first show that $N(v)\cap B(x,w)=\emptyset$ by contradiction. If there exists $u\in N(v)\cap B(x,w)$ we consider the line $l=\li{xu}$. As $x\in l$, we have that $l\notin  L_B \cup L_C\cup L_D$, and since $f(u),v\in l$, we obtain that $l\notin \Li^x_2$. Hence the line $l$ must belong to $\Li^x_3$, and since $w\in l$, we have that $\li{xu}=l=\li{xw}$. The vertex $v$ belongs to $\li{xu}\setminus \li{xw}$, which is a contradiction, thus we conclude  $N(v)\cap B(x,w)=\emptyset$. Let $u\in N(x)\cap N(v)$, therefore, $u\in C(x,w)$. Consider the line $l=\li{uu^*}$. We have that $uu^*\notin E$, hence $l\cap N^3(x)=\emptyset$. Since $x\in l$, the line $l$ must be equal to $\li{xv}\in \Li^x_2$. It follows that $v\in \li{uu^*}$, thus $d(v,u^*)=3$ and $v^*v\notin E$. As $uu^*\notin E$, we obtain that $d(v^*,u)=3$, hence $v^*\in \li{uu^*}=\li{xv}$, yielding a contradiction.
\end{proof}

Given that $E(B(x,w),C(x,w))=E(B(w,x),C(x,w))=\emptyset$, necessarily, $C(x,w)=\emptyset$. Reversing the roles of $x$ and $w$, we conclude that $C(w,x)=\emptyset$. As $L_C=L_D=L_{D'}=\emptyset$, the set $\Li^x_2\cup \Li^x_3\cup L_B$ contains 
$$
|N^2(x)|+ |N^3(x)|+|B(x,w)|-1 =n-2
$$
lines.

\begin{claim}
$G[B(w,x)]$ is a complete graph.
\end{claim}

\begin{proof}
By contradiction, suppose that  $G[B(w,x)]$ is not a complete graph. We may assume that there exists $v\in B(w,x)-v^*$ such that $v^*v\notin E$. Consider the line $l_1=\li{v^*v}$. Since $x\notin l_1$ and $w\in l_1$, the line $l_1$ does not belong to $\Li^x_2\cup \Li^x_3\cup L_B$, therefore $\ell(G)=n-1$ and $\Li(G)=\Li^x_2\cup \Li^x_3\cup L_B\cup \{l_1\}$. We now consider the line $l_2=\li{vu^*}$. Given that $d(v,u^*)=2$ and $x,v^*\in N^2(v)$, we have that $x,v^*\notin l_2$. It follows that $l_2\notin \Li(G)$, which is a contradiction.
\end{proof}

Reversing the roles of $x$ and $w$, we also obtain that $G[B(x,w)]$ is a complete graph. Since it holds that $E(B(x,w),B(w,x))=\{uf(u): u \in B(x,w)\}$, we conclude that $G$ is isomorphic to $M'_{2p}$ with $p=|B(x,w)|+1=|B(w,x)|+1$. As $\Li^x_2\cap \Li^x_3=\emptyset$, the set $\Li(G)$ contains at least two lines. Since the graph $M'_{4}$, the path of length three, has only one line, we obtain that $p\geq 3$, therefore $G$ is isomorphic to $M'_{6}$ or $M'_{8}$. 

This concludes the proof of Proposition \ref{PropRempty}.

\section{Conclusion}
Surprisingly, the core ideas of our proof can be generalized to graphs of arbitrary diameter, though the strength of the argument becomes weaker. The argument of Section \ref{sec:Rempty} can be generalized to obtain the following. Consider a graph $G=(V,E)$ of diameter $k\geq 3$ such that for every vertex $x$ and distinct $i,j\in \{2,3,...,k\}$, it holds that $\Li^x_i\cap \Li^x_j = \emptyset$. Then $G$ satisfies that $\ell(G)\geq |V|-2$, with hope of finding $|V|$ distinct lines when $k\geq 4$. However, this hypothesis appears to be too strong, rendering the other case difficult as the diameter increases. In a graph of diameter $k\geq 3$, the argument of Section \ref{sec:Rnonempty} can be generalized to repair line repetitions in $\Li ^x_{k-1}\cap \Li^x_{k}$ for some vertex $x$, counting lines of the type $\li{yu}$ with $y\in N^k(x)$ and $u\in N(x)$. Nonetheless, when $k\geq 4$, this argument is not enough by itself to count a significant number of lines.

\section*{Acknowledgements}
Both authors were supported by Basal program FB210005, ANID, Chile. Villarroel-Sepúlveda's postgraduate studies were financed by ANID--Subdirección de Capital Humano / Magíster Nacional / 2025 - 22251816.

\bibliographystyle{plain}

\bibliography{biblio}

\end{document}